\documentclass[a4paper,12pt]{amsart}

\usepackage[english]{babel}
\usepackage{amsmath,amssymb,amsfonts,amsthm,enumerate}
\usepackage{url}
\usepackage{graphicx,epstopdf,color}
\usepackage[colorlinks=true, allcolors=blue]{hyperref}
\usepackage{amsrefs}
\usepackage{mathtools}
\usepackage[shortlabels]{enumitem}
\usepackage{appendix}
\usepackage{tikz}
\usepackage{subfig}
\numberwithin{equation}{section}
\allowdisplaybreaks

\newtheorem{theorem}{Theorem}[section]
\newtheorem{lemma}[theorem]{Lemma}
\newtheorem{proposition}[theorem]{Proposition}
\newtheorem{corollary}[theorem]{Corollary}

\theoremstyle{remark}
\newtheorem{definition}[theorem]{Definition}
\newtheorem{remark}[theorem]{Remark}

\newcommand{\R}{\mathbb{R}}

\newcommand{\N}{\mathbb{N}}

\renewcommand{\emptyset}{\varnothing}
\renewcommand{\epsilon}{\varepsilon}
\newcommand{\eps}{\varepsilon}
\newcommand{\F}{\mathcal{F}}

\newcommand{\E}{\mathbf{E}}

\renewcommand{\leq}{\leqslant}

\renewcommand{\geq}{\geqslant}

\newcommand{\per}{\operatorname{Per}}

\newcommand{\dist}{\operatorname{dist}}

\allowdisplaybreaks

\title[Weighted nonlocal area functionals]
{Weighted nonlocal area functionals \\
without the triangle inequality}

\author[S.~Dipierro]{Serena Dipierro}

\author[E.~Valdinoci]{Enrico Valdinoci}

\author[M. Vaughan]{Mary Vaughan}

\address[S. Dipierro, E.~Valdinoci, M.~Vaughan]{The University of Western Australia, 
Department of Mathematics and Statistics,
35 Stirling HWY,
Crawley WA 6009, Australia}
\email{serena.dipierro@uwa.edu.au}
\email{enrico.valdinoci@uwa.edu.au}
\email{mary.vaughan@uwa.edu.au}



\begin{document}

\maketitle


\begin{abstract}
We consider a weighted nonlocal area functional in which the coefficients do not satisfy the triangle inequality. 
In the context of three phase transitions, this means that one of the weights is larger than the sum of the other two, say
$$\sigma_{-1,1} > \sigma_{-1,0} + \sigma_{0,1}.$$
We show that the energy can be reduced by covering interfaces between phases $-1$ and $1$ with a thin strip of phase $0$. 

Moreover, as the fractional parameter $s\nearrow1$,  we prove that the nonlocal energies $\Gamma$-converge to a local area functional with different weights. 

The functional structure of this long-range interaction model is conceptually different from its
classical counterpart, since the functional remains lower semicontinuous, even in the absence
of the triangle inequality.
\end{abstract}

\section{Introduction}
Variational problems often involve multi-phase systems or multi-material interfaces. In this setting, a key structural condition is the so-called ``triangle inequality'', a constraint on the interfacial energy coefficients that ensures energetic stability. More precisely, such a condition guarantees that a direct interface between any two phases is energetically more favorable (or at least not worse) than transitioning through an intermediate phase. This condition permits and often favors direct contact between the extreme phases.

In the absence of the triangle inequality, classical cluster configurations exhibit various pathologies, such as wetting and layering effects, typically arising from the lack of lower semicontinuity in the energy functional (see e.g.~\cite{MR1051228, MR1833000}).

In this paper, we investigate the long-range interaction analogue of this situation. That is, we study a model in which the interfacial energy arises from nonlocal particle interactions, governed by a fat-tailed kernel. The structure of the resulting interfaces differs markedly from the classical case, as the associated energy functional is always lower semicontinuous, even when the triangle inequality is violated.

Nevertheless, in the absence of the triangle inequality,
layering phenomena still occur, favoring intermediate phase transitions. However, in contrast to the local setting, the nonlocal model assigns only quantitatively small energy to thin intermediate layers. 

We also show that the classical local setting emerges as a limit case. Interestingly, the different functional structures in the local and nonlocal frameworks lead to distinct notions of convergence and limiting behavior.
\medskip

More specifically, in this article, we study the weighted nonlocal energy functional 
\begin{equation}\label{eq:energyintro}
\F^s(\{E_{-1},E_0,E_1\}) = \sum_{-1 \leq i <j \leq 1} \sigma_{i,j} \iint_{(E_i \times E_j) \setminus (\Omega^c \times \Omega^c)} \frac{dxdy}{|x-y|^{n+s}},
\end{equation}
where $s \in (0,1)$, $n\geq 1$,  $\Omega \subset \R^n$ is a bounded domain, $\sigma_{i,j} = \sigma_{j,i}>0$, and $\{E_{-1}, E_0, E_1\}$ is a 3-part partition of $\R^n$. 
We primarily focus on when one of the weights, say $\sigma_{-1,1}$, is larger than the sum of the other two, since such energies arise naturally in nonlocal multi-phase coexistence models, see Section~\ref{sec:phase}. 
In this setting, we prove that the energy can be reduced if there is a nontrivial interface between phases $1$ and $-1$ and establish the $\Gamma$-convergence of $(1-s) \F^s$ to a local area functional. 

The energies~\eqref{eq:energyintro} are nonlocal counterparts to immiscible fluid problems~\cites{Almgren, Leonardi,White} in which three immiscible fluids $E_1,E_0,E_{-1}$ occupy the container $\Omega$. 
Indeed, each integral can be seen as the nonlocal area of an interface between the two disjoint sets $E_i$ and $E_j$ for $i\not=j$ and $\sigma_{i,j}$ is the surface tension between fluids $E_i$ and $E_j$.
 
In the seminal paper~\cite{CRS}, Caffarelli, Roquejoffre, and Savin introduced the nonlocal perimeter of a single measurable set by taking into account long-range interactions between the set and its complement. This corresponds to a 2-part partition of $\R^n$ with energy
\begin{equation}\label{eq:perimintro}
\per_{\Omega}^s(E) = \iint_{(E \times E^c) \setminus (\Omega^c \times \Omega^c)} \frac{dxdy}{|x-y|^{n+s}}. 
\end{equation}

We define the nonlocal area between two disjoint sets $A$ and $B$, not necessarily forming a 2-part partition, by considering only the long-range interactions between $A$ and $B$.  

\begin{definition}\label{defn:ij-per}
For a bounded domain $\Omega \subset \R^n$ and disjoint, measurable sets $A, B \subset \R^n$,  define
\[
\per^s_\Omega(A,B) := \iint_{( A \times B) \setminus (\Omega^c \times \Omega^c)} \frac{dxdy}{|x-y|^{n+s}}.
\]
\end{definition}

The special case $B= A^c$ corresponds to the nonlocal perimeter of $A$ in~$\Omega$:
\[
\per^s_\Omega(A)= \per^s_\Omega(A,A^c).
\] 
On the other hand, if $|\R^n \setminus (A \cup B)
|>0$, then 
$\per^s_\Omega(A,B)$ ignores all interactions coming from $\R^n \setminus (A \cup B)$. Notice also that $\per_{\Omega}^s(A,B)>0$ as long as $A,B\not= \emptyset$, even if the two sets do not share an interface. Further properties will be discussed in Sections~\ref{sec:prelim} and~\ref{sec:uniformestimates}. 

\medskip

Let~$\Omega \subset \R^n$ be a bounded domain with Lipschitz boundary. 
We say that $\E := \{E_{-1},E_0, E_1\}$ is \emph{admissible} if it forms a 3-part partition of $\R^n$ (also called a 3-cluster), namely
\[
\left|\R^n \setminus \bigcup_{-1 \leq i \leq 1} E_i\right| = 0 \quad \hbox{and} \quad |E_i \cap E_j| = 0~\hbox{for}~i\not=j.
\]
Then, we can write~\eqref{eq:energyintro} as 
\begin{equation}\label{eq:sigma-energy}
\F^s(\E)
	= \sum_{-1 \leq i < j \leq 1} \sigma_{i,j}\per^s_\Omega(E_i, E_j) \quad \hbox{for}~\E~\hbox{admissible},
\end{equation}
where $\sigma_{i,j}= \sigma_{j,i}>0$ is the surface tension at the interface of~$E_i$ and~$E_j$.

We write $\F^1(\E)$ to include the classical case (corresponding to~$s=1$) in which $\per^1_\Omega(E_i, E_j)=\per_\Omega(E_i, E_j)$ is the classical area in~$\Omega$ of the interface separating $E_i$ and $E_j$. Similarly, we write $\per^1_{\Omega}(E) = \per_{\Omega}(E)$ to denote the classical perimeter of $E$ in~$\Omega$, see~\cite{Giusti}.  

\medskip

We now present our main results. 

\subsection{The triangle inequality}

Our first key observation is that $\F^s$ is lower semicontinuous. 

We say that $\E^k \to \E$  in $L^1_{\text{loc}}(\R^n)$ if $\chi_{E_i^k} \to \chi_{E_i}$ in $L^1_{\text{loc}}(\R^n)$ as $k \to \infty$ for each $-1 \leq i \leq 1$. 

\begin{proposition}[Lower semicontinuity]\label{prop:lsc}
For $s \in (0,1)$, the functional $\mathcal{F}^s$ is lower semicontinuous. In particular, for admissible $\E^k$ and $\E$, if $\E^k \to \E$  in $L^1_{\text{loc}}(\R^n)$, then
\[
\liminf_{k \to \infty} \F^s(\E^k) \geq \F^s(\E).
\]
\end{proposition}

The result in Proposition~\ref{prop:lsc} is technically simple,
but conceptually deep, since it highlights an important structural difference with respect to
the classical case, which impacts also the quantitative
analysis of the energy competitors, the pointwise limit of the energy functional
as~$s\nearrow1$, as well as its limit in the sense of~$\Gamma$-convergence.

Indeed,
the fact that there is no additional assumption on the surface tensions $\sigma_{i,j}>0$ is in stark contrast to the classical case. 
As a matter of fact, it was
 observed by White in~\cite{White} and by Ambrosio--Braides in~\cite{AmbrosioBraides}
 that the local energy $\F^1$ is lower semicontinuous if and only if the  coefficients satisfy a so-called ``triangle inequality,'' see also~\cite{Morgan}.
 
For this, define $\alpha_i$ for $-1 \leq i \leq 1$ by 
\begin{equation}\label{eq:alpha}
\begin{aligned}
\alpha_{-1} &:= \frac12 ( \sigma_{-1,1} + \sigma_{-1,0} - \sigma_{0,1}) ,\\
\alpha_{0} &:=  \frac12(\sigma_{-1,0} + \sigma_{0,1} - \sigma_{-1,1}), \\
\alpha_{1} &:= \frac12(\sigma_{-1,1} + \sigma_{0,1} - \sigma_{-1,0}). 
\end{aligned}
\end{equation}
We say that
\begin{center}
$\{\sigma_{i,j}\}_{-1 \leq i<j\leq1}$ satisfies the \emph{triangle inequality} if $\alpha_i\geq 0$ for all $-1 \leq i \leq 1$.
\end{center}

With $\alpha_i$ given in~\eqref{eq:alpha}, one can check (with Lemma~\ref{lem:JAB}) that~\eqref{eq:sigma-energy} is equivalently given by 
\begin{equation}\label{eq:alpha-energy}
\F^s(\E)
	=  \sum_{-1 \leq i \leq 1} \alpha_i \per^s_\Omega(E_i). 
\end{equation}
Energies \emph{defined} as~\eqref{eq:alpha-energy} with positive weights $\alpha_i>0$ have been studied in regards to nonlocal clusters, see~\cites{Cesaroni-Novaga,Colombo-Maggi}.
In our setting, we emphasize that the coefficients $\alpha_i$ in~\eqref{eq:alpha-energy} cannot be viewed as ``arbitrary'' since they are uniquely determined by the $\sigma_{i,j}$. 
For instance, the relation
\begin{equation}\label{eq:alphasums}
\alpha_i + \alpha_j =  \sigma_{i,j} >0 \quad \hbox{for all}~i\not=j
\end{equation}
implies that at most one of the $\alpha_i$ can be nonpositive. 

We stress that the long-range interactions provide significant quantitative
and qualitative differences regarding the structure of the corresponding clusters. To better understand these differences, it will be useful to revisit the classical case
under a ``nonlocal'' perspective.
To this end,
we now describe in more detail White's observations in~\cite{White}*{Section 2}
for classical clusters. 

Suppose that $\alpha_0<0$. 
If an admissible~$\E$ has a nontrivial (Lipschitz)  interface between phases 
 -1 and 1, then one can cover it with a thin strip of phase 
 $0$ of width $\eps>0$ and the corresponding $\E^\eps$ has strictly smaller energy, see Figure~\ref{fig:introstrip}. Moreover, for $\eps>0$ sufficiently small, the quantity $\F^1(\E^\eps)$ can be made arbitrarily close to $\F^*(\E)$ which is given by
\begin{equation}\label{eq:F*}
\F^*(\E) = \sum_{-1 \leq i <j\leq 1} \sigma_{i,j}^* \per_\Omega (E_i, E_j)
\end{equation}
with the new surface tensions $\sigma_{i,j}^*>0$  given by
\[
\sigma_{-1,1}^* = \sigma_{-1,0} + \sigma_{0,1}, \quad \sigma_{0,1}^* = \sigma_{0,1}, \quad \hbox{and} \quad \sigma_{-1,0}^* = \sigma_{-1,0}.
\] 
In particular, 
\begin{equation}\label{eq:f1fstar}
\F^1(\E) = \F^*(\E) - 2\alpha_0 \per_{\Omega}(E_1, E_{-1}),
\end{equation}
and, since $\per_{\Omega}(E_1^\eps, E_{-1}^\eps) = 0$, 
one obtains 
\begin{equation}\label{eq:white}\begin{split}
&\F^1(\E) - \F^1(\E^\eps)
	= \F^1(\E) - \F^*(\E^\eps)
\\&\qquad=
\F^*(\E) - \F^*(\E^\eps)- 2\alpha_0 \per_{\Omega}(E_1, E_{-1})
\geq 2|\alpha_0| \per_{\Omega}(E_1, E_{-1}) - 2C(\sigma_{0,1} + \sigma_{-1,0})\eps
\end{split}
\end{equation}
for some $C>0$, independent of $\eps$.
Consequently, the difference in energy in~\eqref{eq:white} is of order 1 and $\F^1$ is not lower semicontinuous. \medskip

\begin{figure}[hbt]
\begin{center}
\subfloat[An admissible $\E$]{
\label{fig:physical-1}
\begin{tikzpicture}[scale=3]
\draw[magenta,thick] (0,1.2)--(0,0);
\draw[cyan,thick] (1.039,-.6)--(0,0);
\draw[blue,thick] (-1.039,-.6)--(0,0);
\draw[thick] (0,0) circle (1cm);
\node[] at (0,-.5) { $E_0$};
\node[] at (.433,.25) { $E_{-1}$};
\node[] at (-.433,.25) { $E_{1}$};
\node[] at (1.0825,.625) {\large $\Omega$};
\node[magenta] at (.15,.77) {\scriptsize $\sigma_{-1,1}$};
\node[cyan] at (.72,-.31) {\scriptsize $\sigma_{-1,0}$};
\node[blue] at (-.72,-.32) {\scriptsize $\sigma_{0,1}$};
\end{tikzpicture}
}
\hspace{1.5cm}
\subfloat[The corresponding $\E^\varepsilon$]{
\label{fig:physical-i}
\begin{tikzpicture}[scale=3]
\draw[cyan,thick] (0.07 ,.998)--(.07,0);
\draw[cyan,thick] (0,0)--(.07,0);
\draw[blue,thick] (-0.07,.998)--(-.07,0);
\draw[blue,thick] (-.07,0)--(0,0);
\draw[cyan,thick] (1.039,-.6)--(0,0);
\draw[blue,thick] (-1.039,-.6)--(0,0);
\draw[thick] (0,0) circle (1cm);
\node[] at (0,-.5) { $E_0^\varepsilon$};
\node[] at (.433,.25) { $E_{-1}^\varepsilon$};
\node[] at (-.433,.25) { $E_{1}^\varepsilon$};
\node[] at (1.0825,.625) {\large $\Omega$};
\node[] at (0,.6) { \scriptsize $E_0^\varepsilon$};
\draw[<->] (-.065,.2)--(.065,.2);
\node[] at (0,.24) {\scriptsize $\varepsilon$};
\draw[magenta,thick] (0,1)--(0,1.2);
\draw[magenta,thin, densely dotted] (0,1)--(0,.67);
\draw[magenta,thin, densely dotted] (0,.54)--(0,.29);
\draw[magenta,thin, densely dotted] (0,.16)--(0,0);
\node[cyan] at (.72,-.31) {\scriptsize $\sigma_{-1,0}$};
\node[cyan] at (.22,.77) {\scriptsize $\sigma_{-1,0}$};
\node[blue] at (-.72,-.32) {\scriptsize $\sigma_{0,1}$};
\node[blue] at (-.19,.77) {\scriptsize $\sigma_{0,1}$};
\end{tikzpicture}
}
\end{center}
\caption{Covering an interface $\partial E_1 \cap \partial E_{-1} \cap \Omega$ with an $\eps$-strip of phase $0$}
\label{fig:introstrip}
\end{figure}
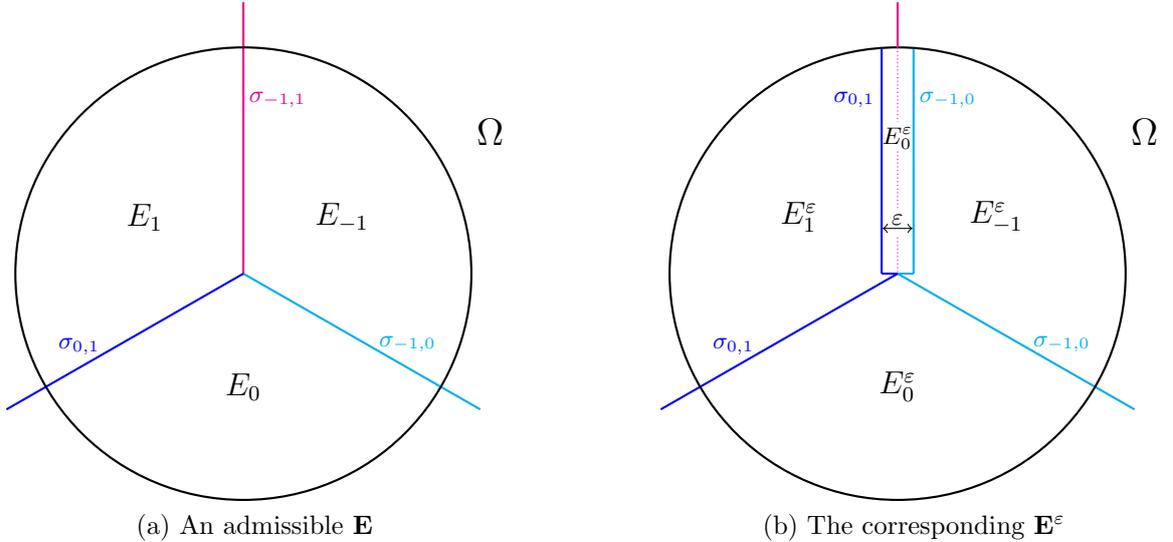

With respect to this construction, at a first glance, the functional discrepancy 
in the absence of a triangular
inequality
between the
nonlocal case (which is lower semicontinuous) and the classical counterpart
(which is not lower semicontinuous) may seem problematic or even contradictory,
since one may wonder what is the nonlocal analogue of
the ``energy bubble''
lost in the classical case. To appreciate this, a quantitative analysis is needed
to compare the energy of the competitor~$\E^\varepsilon$
obtained from an~$\varepsilon$-strip of the intermediate phase
in the nonlocal case
and to compare this result with the classical
counterpart. We will see that 
in the nonlocal case the energy bubble produced with this procedure is of order~$\varepsilon^{1-s}$ and this is both
compatible with the
lower semicontinuity of the energy functional (since the energy of this bubble vanishes as~$\varepsilon\searrow0$) and consistent with the lack of semicontinuity of the classical case
(since the energy of the bubble becomes independent of~$\varepsilon$ as~$s\searrow1$).

Namely, when $s \in (0,1)$, we consider the same admissible $\E$ and $\E^{\eps}$
as in the classical case, but
unlike in the local problem, there are long-range interactions between the sets $E_1^\eps$ and $E_{-1}^\eps$, so that, at least formally,
\[
\lim_{\eps \searrow 0} [\per_{\Omega}^s(E_1, E_{-1}) - \per_{\Omega}^s(E_1^\eps, E_{-1}^\eps)] =0.
\]
In this way, we prove that the nonlocal energy of $\E^\eps$ is strictly smaller than that of $\E$,
even in the nonlocal case, but the difference in energy is of order $\eps^{1-s}$, which is negligible if (and only if) $s\in(0,1)$.
 
We give an informal statement of the result here; see Theorem~\ref{thm:reduce} for the precise statement.

\begin{theorem}\label{thm:reduceintro}
Assume that $\alpha_0<0$ and $s \in (0,1)$. If $\partial E_1 \cap \partial E_{-1} \cap K$ is Lipschitz for some compact $K \Subset \Omega$ with small Lipschitz norm, then there are constants $C>0$ and $0 < \eps_0<1$, such that, for all $0  < \eps < \eps_0$, there exists an admissible~$\E^\eps$ satisfying
\begin{equation}\label{eq:sdifference}
\F^s(\E) - \F^s(\E^\eps) \geq C \eps^{1-s}. 
\end{equation}
\end{theorem}

Notice that sending $s \nearrow 1$ in~\eqref{eq:sdifference} is consistent with~\eqref{eq:white}. However, when~$s\in(0,1)$,
\[
\lim_{\eps \searrow 0} [\F^s(\E) -\F^s(\E^\eps)] \geq 0
\]and therefore
there is no contradiction to the lower semicontinuity in Proposition~\ref{prop:lsc}. 

\medskip

We say that an admissible~$\E$ is a \emph{minimizer} for $\F^s$  in~$\Omega$ if for any admissible $\widetilde{\E}$ with $\widetilde{E}_i \cap \Omega^c = E_i \cap \Omega^c$ for each $-1 \leq i \leq 1$, we have
\begin{equation}\label{eq:min}
\F^s(\E) \leq \F^s(\widetilde{\E}). 
\end{equation}
We say that $\E$ is a \emph{local minimizer} if~\eqref{eq:min} holds for any $\widetilde{\E}$ satisfying $\widetilde{E}_i \cap (\Omega')^c = E_i \cap (\Omega')^c$ for all $\Omega' \Subset \Omega$ and each $-1 \leq i \leq 1$.
We can similarly define the notion of minimizers and local minimizers of $\F^*$ (and $\F^1$) in~$\Omega$.

In light of Theorem~\ref{thm:reduceintro}, one expects that local minimizers $\E$ of $\F^s$ are such that there is no transition directly from phase~$-1$ to phase~$1$; namely, one must pass nontrivially through phase~$0$.
We leave the question of regularity as future work.

\subsection{Convergence as $s\nearrow1$}

Let~$\omega_n$ denote the measure of the unit ball in $\R^n$, $n \geq 1$, and~$\omega_0 = 1$. To emphasize the dependence on the domain, we will sometimes use the notation, for~$s \in (0,1]$,
\begin{equation}\label{eq:Fomega}
\F^*(\E,\Omega) = \sum_{-1 \leq i<j \leq 1} \sigma_{i,j}^* \per_\Omega(E_i,E_j), \qquad
\F^s(\E,\Omega) = \sum_{-1 \leq i<j \leq 1} \sigma_{i,j} \per_\Omega^s(E_i,E_j).
\end{equation}

We prove that the renormalized energy $(1-s)\mathcal{F}^s(\E)$ converges to $\omega_{n-1}\F^1(\E)$ as $s \nearrow 1$ for a given admissible $\E$.  The normalization $(1-s)$ was established in~\cite{Brezis}, see also~\cite{ADPM,  CVcalcvar}. 

\begin{theorem}\label{thm:uniformintro} 
Assume that $\Omega \subset \R^n$ is a bounded domain with Lipschitz boundary. Let~$\E$ be admissible such that 
$\partial E_i \cap \Omega$ is Lipschitz for $-1 \leq i \leq 1$. 

Then, for any $\Omega' \Subset \Omega$, it holds that
\[
\lim_{s \nearrow 1} \bigg| (1-s)\F^s(\E,\Omega')-\omega_{n-1} \F^1(\E,\Omega')\bigg| \leq 2 \omega_{n-1}\F^1(\E, \partial \Omega'). 
 \]
\end{theorem}

\begin{remark}\label{rem:uniform}
If $\alpha_0\leq0$ and $\E$ is such that $\per_{\overline{\Omega}}(E_1,E_{-1}) = 0$, then by~\eqref{eq:f1fstar}, Theorem~\ref{thm:uniformintro} holds with $\F^*$ in place of $\F^1$. 

Also, when a ``transversality'' condition ensures that~$\F^1(\E, \partial \Omega')=0$,
Theorem~\ref{thm:uniformintro} returns that
$$\lim_{s \nearrow 1} (1-s)\F^s(\E,\Omega')=\omega_{n-1} \F^1(\E,\Omega').$$
\end{remark}

Since $\sigma_{i,j}>0$, Theorem~\ref{thm:uniformintro} follows from a similar statement for $\per_{\Omega}^s(A,B)$, see Theorem~\ref{thm:CVAB}. 
More precise estimates are written in Section~\ref{sec:uniformestimates}. 

We stress that in Theorem~\ref{thm:uniformintro} there is no additional assumption on $\sigma_{i,j}$ regarding the triangle inequality.
This is interesting, since the limit functional~$\F^{1}$ is not necessarily lower semicontinuous
and therefore the ``pointwise'' limit described in Theorem~\ref{thm:uniformintro} may
differ from the corresponding $\Gamma$-limit.

This discrepancy is addressed in the following result.
Indeed, if $\alpha_0\leq0$, then $\F^*$ is the largest lower semicontinuous functional that is less than or equal to $\F^1$, and we prove that $(1-s)\F^s$ $\Gamma$-converges to $\omega_{n-1}\F^*$ as $s \nearrow 1$, according to the forthcoming statement. 

\begin{theorem}\label{thm:s-gamma}
Let~$\Omega \subset \R^n$ be a bounded domain with Lipschitz boundary and assume that~$\alpha_0\leq0$. Then $(1-s)\F^s$ $\Gamma$-converges to $\omega_{n-1}\F^*$.

In particular, if $s_k \nearrow1$ as $k \to \infty$, then the following hold:
\begin{enumerate}
\item If the admissible $\E^k$, $\E$  are such that $\E^k \to \E$ in $L^1_{\text{loc}}(\R^n)$, then
\[
\liminf_{k \to \infty}(1-s_k) \F^{s_k}(\E^k) \geq \omega_{n-1}\F^*(\E). 
\]
\item If $\E$ is admissible, then there exists admissible $\E^k$ such that $\E^k \to \E$ in $L^1_{\text{loc}}(\R^n)$ and
\[
\limsup_{k \to \infty}(1-s_k) \F^{s_k}(\E^k) \leq \omega_{n-1}\F^*(\E). 
\]
\end{enumerate}
\end{theorem}

We recall that
Ambrosio--De Phillippis--Martinazzi~\cite{ADPM} proved the $\Gamma$-convergence of $(1-s) \per_{\Omega}^s(E)$ to $\omega_{n-1}\per_{\Omega}(E)$. 
For nonlocal clusters defined as in~\eqref{eq:alpha-energy}, Cesaroni--Novaga~\cite{Cesaroni-Novaga}*{Theorem 2.10} proved $\Gamma$-convergence under the assumption that $\alpha_i>0$ for all $i$. 
In contrast, our work allows for $\alpha_i \leq 0$ as well. 

\medskip   

We also prove a refinement of Theorem~\ref{thm:s-gamma} with given external data (see Propositions~\ref{prop:bdry-liminf} and~\ref{prop:bdry-limsup} and, in dimension one, see Theorem~\ref{thm:1D-gamma}). In turn, we prove that  minimizers of $\F^s$ converge to minimizers of $\F^*$. 

For $\delta>0$, we set
\begin{equation}\label{eq:delta-nbhd}
\begin{aligned}
 \Omega_\delta &=  \{x \in \R^n~\hbox{s.t.}~\dist(x, \partial \Omega)<\delta\},\\
 \Omega_\delta^+ &=  \{x \in \Omega^c~\hbox{s.t.}~\dist(x, \partial \Omega)<\delta\},\\
  \Omega_\delta^- &=  \{x \in \Omega~\hbox{s.t.}~\dist(x, \partial \Omega)<\delta\},
\end{aligned}
\end{equation}
where
$$ \dist(x, \partial \Omega):=\inf_{y\in\partial\Omega}|x-y|.$$

We say that a measurable set $E$ is transversal to $\partial \Omega$ if 
\[
\lim_{ \delta \searrow 0^+} \per_{\Omega_\delta}(E) = 0
\]
and that $E$ is transversal to $\partial \Omega^+$ if 
\[
\lim_{ \delta \searrow 0^+} \per_{\Omega_\delta^+}(E) = 0.
\]
We say that $\E$ is transversal to $\partial \Omega$ (respectively  to $\partial \Omega^+$) if each $E_i$ is transversal to $\partial \Omega$ (respectively to $\partial \Omega^+$). 

An admissible $\E$ is \emph{polyhedral in~$\Omega$} if, for each $E_i$, there is a finite number of $(n-1)$-dimensional simplexes $T_1, \dots T_{r_i} \subset \R^n$ such that $\partial E_i$ coincides with $\cup_j(T_j \cap \Omega)$ up to a set of $\mathcal{H}^{n-1}$ dimension zero.

\begin{theorem}\label{thm:localmin}
Let~$\Omega \subset \R^n$ be a bounded domain with $C^1$ boundary and assume that $\alpha_0\leq 0$. 
Consider an admissible~$\widetilde{\E}$ that is polyhedral in $\Omega^c$ and transversal to $\partial \Omega^+$. 

Let~$s_k \nearrow 1$ as $k \to \infty$.
Assume that $\E^k$ are local minimizers of $\F^{s_k}$ in~$\Omega$ such that $E_i^k \cap \Omega^c = \widetilde{E}_i \cap \Omega^c$  for each $-1 \leq i \leq 1$. 

Then there is an admissible~$\E$ such that
$E_i \cap \Omega^c = \widetilde{E}_i \cap \Omega^c$ for each $-1 \leq i \leq 1$,
$\E$ is a minimizer of $\F^*$, and, up to a subsequence, it holds that $\E^k \to \E$ in $L^1_{\text{loc}}(\R^n)$ as $k \to \infty$ and 
\begin{equation}\label{eq:recoverenergy}
\lim_{k \to \infty} (1-s_k) \F^{s_k}(\E^k) = \omega_{n-1} \F^*(\E).
\end{equation}
\end{theorem}

\subsection{Organization of the paper}

The rest of the paper is organized as follows. 
In Section~\ref{sec:phase}, we present an application of the energies $\F^s$ for which the triangle inequality is not satisfied. This is interesting, since it shows
that long-range area functionals violating the triangle inequality arises
naturally from nonlocal phase transition models.

Then,
we discuss preliminary properties of nonlocal area functionals and minimizers of $\F^s$ in Section~\ref{sec:prelim}. 
Theorem~\ref{thm:reduceintro} is made precise and proven in Section~\ref{sec:reduce}. 
We prove Theorems~\ref{thm:s-gamma} and~\ref{thm:localmin} in Section~\ref{sec:gamma}. 
Finally, Appendix~\ref{appendix:1D} contains the one-dimensional counterparts of our main results.

\subsection{Notation}\label{sec:notation}

For measurable sets $A,B \subset \R^n$, let
\[
L(A,B)=L^s(A,B):= \iint_{A \times B} \frac{dx dy}{|x-y|^{n+s}}, \quad s \in (0,1). 
\]
Consequently, if $|A \cap B|=0$, one can write
\[
\per^s_\Omega(A,B) := L(A \cap \Omega, B \cap \Omega) + L(A \cap \Omega, B \cap \Omega^c) + L(A\cap \Omega^c, B \cap \Omega).
\]

\section{Application to multi-phase transition problems}\label{sec:phase}

In this section, we present a concrete application to multi-phase transitions that gives rise to an energy~\eqref{eq:sigma-energy} for which the triangle inequality is \emph{not} satisfied. 

Consider a long-range multi-phase coexistence energy functional given by 
\begin{equation}\label{eq:Feps}
\mathcal{E}_\eps(u) := \frac12 \iint_{(\R^n \times \R^n) \setminus (\Omega^c \times \Omega^c)} \frac{|u(x) - u(y)|^2}{|x-y|^{n+s}} \, dx\,dy
+\frac{1}{\eps^{s}} \int_{\Omega} W(u(x)) \, dx,
\end{equation}
where~$\Omega \subset \R^n$
is a bounded domain, $W$ is a triple-well potential,~$\eps>0$ is a small parameter, and $s \in (0,1)$. 
More precisely, we assume that $W \in C^2(\R)$ satisfies
\begin{equation}\label{eq:W}
\begin{cases}
W(-1) = W(0) = W(1) = 0,\\
W'(-1) = W'(0) = W'(1) =0,\\
W(r)>0 \quad \hbox{for any}~r \in (-1,0) \cup (0,1), \\
W''(-1),~W''(0),~W''(1)>0.
\end{cases}
\end{equation}
The potential energy in~\eqref{eq:Feps}, which favors the pure phases indexed by $i \in \{-1,0,1\}$, is balanced by the long-range interaction energy that takes into account long-range particle interactions. 

In the nonlocal setting two-phase transition problems in which $W$ is a double-well potential have been considered in~\cites{SV-gamma, SV-dens, MR1612250} and more recently in~\cites{PAPERONE, PAPERTWO, MR4743481, 2024arXiv241101586S}. 
Exactly as in the proof of~\cite{SV-gamma}*{Theorem 1.2}, one can show that $\mathcal{E}_\eps$ $\Gamma$-converges to a nonlocal perimeter-type functional when $u \in \{-1,0,1\}$. 

\begin{theorem}
Let~$s \in (0,1)$ and $\Omega \subset \R^n$ be a bounded domain. Then $\mathcal{E}_\eps$ $\Gamma$-converges to $\mathcal{E}$ which is given by 
\begin{equation}\label{eq:gamma-limit-F}
\mathcal{E}(u) = \begin{cases}
\displaystyle \frac12 \iint_{(\R^n \times \R^n) \setminus (\Omega^c \times \Omega^c)} \frac{|u(x) - u(y)|^2}{|x-y|^{n+s}} \, dy \, dx
	 &\displaystyle \hbox{if}~u\big|_{\Omega \cap E_i} = i \in \{-1,0,1\},\\
+\infty & \hbox{otherwise}, 
\end{cases}
\end{equation}
where $\E = \{E_{-1}, E_0, E_{1}\}$ is admissible. 
\end{theorem}

We can express the energy functional $\mathcal{E}$ in terms of nonlocal perimeter functionals. 

\begin{lemma}\label{lem:F-perimeter}
Let~$s \in (0,1)$ and $\Omega \subset \R^n$ be a bounded domain. 
Suppose that $\E$ is admissible and let $u :\R^n \to \{-1,0,1\}$ be such that $\mathcal{E}(u)<\infty$. If $E_i = \{u=i\}$ for $-1 \leq i \leq 1$, then
\begin{equation}\label{eq:gamma-limit}
\frac12 \iint_{(\R^n \times \R^n) \setminus (\Omega \times \Omega)} \frac{|u(x) - u(y)|^2}{|x-y|^{n+s}} \, dy \, dx
	= \per^s_\Omega(E_{-1}, E_0) 
		+  \per^s_\Omega(E_{0}, E_1)
		+ 4  \per^s_\Omega(E_{-1}, E_1).
\end{equation}
\end{lemma}

\begin{proof}
For measurable sets $A,B \subset \R^n$, we use the notation
\[
u(A,B) = \iint_{A \times B}\frac{|u(x) - u(y)|^2}{|x-y|^{n+s}} \, dy \, dx, 
\]
so that
\[
\frac12 \iint_{(\R^n \times \R^n) \setminus (\Omega^c \times \Omega^c)} \frac{|u(x) - u(y)|^2}{|x-y|^{n+s}} \, dy \, dx=\frac12 u(\Omega,\Omega) + u(\Omega, \Omega^c).
\]
Since $u = i$ in $E_i$, we have, for $i\not=j$, 
\begin{eqnarray*}&&
u(E_i \cap \Omega, E_i \cap \Omega) = 0 ,\\ &&
u(E_i \cap \Omega, E_j \cap \Omega) = (j-i)^2 L(E_i \cap \Omega, E_j \cap \Omega), \\ &&
u(E_i \cap \Omega, E_j \cap \Omega^c) = (j-i)^2 L(E_i \cap \Omega, E_j \cap \Omega^c). 
\end{eqnarray*}
Hence, it holds that
\begin{align*}
\frac12 u(\Omega,\Omega) 
	&= u(E_{-1} \cap \Omega,E_0 \cap \Omega) + u(E_{-1} \cap \Omega,E_1 \cap \Omega) 
		+u(E_0 \cap \Omega,E_1 \cap \Omega)
\\&= L(E_{-1} \cap \Omega,E_0 \cap \Omega)
	+ 4 L(E_{-1} \cap \Omega,E_1 \cap \Omega)
	+ L(E_{0} \cap \Omega,E_1 \cap \Omega),
\end{align*}
and
\begin{align*}
u(\Omega, \Omega^c)
	&=   u(E_{-1} \cap \Omega,E_0 \cap \Omega^c)
		+  u(E_{-1} \cap \Omega,E_1 \cap \Omega^c)\\
	&\qquad + u(E_{0} \cap \Omega,E_{-1} \cap \Omega^c)
		+  u(E_{0} \cap \Omega,E_{1} \cap \Omega^c)\\
	&\qquad + u(E_{1} \cap \Omega,E_{0} \cap \Omega^c)
		+  u(E_{1} \cap \Omega,E_{-1} \cap \Omega^c)\\
	&=   L(E_{-1} \cap \Omega,E_0 \cap \Omega^c)
		+  4L(E_{-1} \cap \Omega,E_1 \cap \Omega^c)\\
	&\qquad + L(E_{0} \cap \Omega,E_{-1} \cap \Omega^c)
		+  L(E_{0} \cap \Omega,E_{1} \cap \Omega^c)\\
	&\qquad + L(E_{1} \cap \Omega,E_{0} \cap \Omega^c)
		+  4L(E_{1} \cap \Omega,E_{-1} \cap \Omega^c). 
\end{align*}
The result follows. 
\end{proof}

We stress that
the coefficients $\sigma_{i,j} = (j-i)^2$, $-1 \leq i < j \leq 1$ on the right-hand side of~\eqref{eq:gamma-limit} do not satisfy the triangle inequality, since
\begin{align*}
\alpha_{-1} = \alpha_1 = 2 >0 \quad \hbox{and} \quad \alpha_0 = -1<0. 
\end{align*}

This example is interesting, since it shows that clusters violating the triangle inequality arise
naturally in long-range interaction models.

More generally, if the potential $W$ has three wells at some $i<j<k$, then the limiting functional is
\[
(j-i)^2 \per^s_\Omega(E_i, E_j)
	+ (k-j)^2 \per^s_\Omega(E_j, E_k)
	+ (k - i)^2 \per^s_\Omega(E_i, E_k),
\]
and the triangle inequality for the coefficients does not hold since
\begin{align*}
\alpha_j = \frac12 \left[
(j-i)^2+(k-j)^2 - (k-i)^2 \right]
= -(j-i)(k-j)<0.
\end{align*}

\section{Preliminaries}\label{sec:prelim}

In this section, we establish some preliminary results and identities regarding  the nonlocal functionals $\per_\Omega^s(A)$ and $\per_\Omega^s(A,B)$ and regarding minimizers of the energy $\F^s$ for $s \in (0,1)$.

\subsection{Identities for nonlocal area functionals}

We first summarize some helpful identities for $\per_\Omega^s(A,B)$ and $\per_\Omega^s(E)$. 
It is clear from the definition that
\[
\per_{\Omega}^s(A,B) = \per_{\Omega}^s(B,A) \geq 0. 
\]
We establish a relationship between the two area functionals.

\begin{lemma}\label{lem:JAB}
If $A, B \subset \R^n$ are such that $|A \cap B| = 0$, then 
\[
\per_{\Omega}^s(A) = \per_\Omega^s(A,B) + \per_\Omega^s (A, \R^n \setminus (A \cup B))
\]
and
\[
\per^s_\Omega(A,B) = \frac{\per^s_\Omega(A) + \per^s_\Omega(B) - \per^s_\Omega(\R^n \setminus (A\cup B))}{2}.
\]
\end{lemma}

\begin{proof}
For ease in the proof, let $E:= \R^n \setminus (A\cup B)$ so that $E^c = A \cup B$. We have that
\begin{align*}
\per^s_\Omega(A) 
&= L(A \cap \Omega, (B\cup E) \cap \Omega) + L(A \cap \Omega, (B \cup E) \cap \Omega^c) + L(A \cap \Omega^c, (B \cup E) \cap \Omega)\\
&=  L(A \cap \Omega, B \cap \Omega) + L(A \cap \Omega, B \cap \Omega^c) + L(A \cap \Omega^c, B \cap \Omega)\\
&\qquad + L(A \cap \Omega, E \cap \Omega) + L(A \cap \Omega, E \cap \Omega^c) + L(A \cap \Omega^c, E \cap \Omega)\\
&= \per^s_\Omega(A,B) + \per_{\Omega}^s(A,E),
\end{align*} 
which is the first identity. The second identity follows from the first. 
\end{proof}

\begin{remark}\label{rem:JAB1}
Lemma~\ref{lem:JAB} also holds in the classical setting, namely when $s=1$. 
\end{remark}

The next two lemmata can be checked by direct computation. We state them for later reference.  

\begin{lemma}\label{lem:E-difference}
Let~$A \subset \Omega$ and $E \subset \R^n$ be measurable sets. 
\begin{enumerate}
\item If $|E \cap A|>0$, then
\begin{align*}
\per_\Omega^s(E) -\per_\Omega^s(E \cap A^c)
	&=  L(E \cap A, E^c) - L(E \cap A, E \cap A^c).
\end{align*}
\item If $|E \cap A| =0$, then
\[
\per_\Omega^s(E \cup A)-\per_\Omega^s(E)
	= L( A, E^c \cap A^c)-L(E,  A).
\]
\end{enumerate}
\end{lemma}

\begin{lemma}\label{lem:omegaomega}
Let~$A,B \subset \R^n$ be such that $|A \cap B|=0$ and let $\Omega' \subset \Omega$. It holds that
\[
\per_\Omega^{s}(A) - \per_{\Omega'}^{s}(A)
	= L(A \cap (\Omega\setminus \Omega'), A^c \cap (\Omega')^c)+ L(A \cap \Omega^c, A^c \cap(\Omega\setminus \Omega'))
\]
and 
\[
\per_\Omega^{s}(A,B) - \per_{\Omega'}^{s}(A,B)
	= L(A \cap (\Omega \setminus \Omega'), B \cap (\Omega')^c)
		+ L(A \cap \Omega^c, B \cap (\Omega \setminus \Omega')). 
\]
\end{lemma}

\subsection{Minimizers of $\F^s$}

Now we establish some properties of minimizers of $\F^s$. 
We begin by proving Proposition~\ref{prop:lsc}. 

\begin{proof}[Proof of Proposition~\ref{prop:lsc}]
Suppose that $\chi_{A^{k}} \to \chi_A$ and $\chi_{B^k} \to \chi_B$ in $L_{\text{loc}}^1(\R^n)$. It was observed in the proof of~\cite{CRS}*{Proposition 3.1} that there is a subsequence, indexed by $k_\ell$, such that
\[
\liminf_{\ell \to \infty} L(A^{k_\ell}, B^{k_\ell}) \geq L(A,B).
\]
Recalling Definition~\ref{defn:ij-per}, we have
\[
\liminf_{\ell \to \infty} \per_\Omega^s (E_{i}^{k_\ell}, E_j^{k_\ell}) \geq \per_\Omega^s (E_{i}, E_j) \quad \hbox{for}~i\not=j
\]
and the result follows from~\eqref{eq:sigma-energy} since $\sigma_{i,j}>0$. 
\end{proof}

Following the ideas in~\cite{CRS}, we obtain the following existence and compactness results. 

\begin{theorem}[Existence]
Given an admissible~$\mathbf{F}$,
there exist an admissible~$\E$  such that $E_i \cap \Omega^c = F_i \cap \Omega^c$ for $-1 \leq i \leq 1$ and
\[
\F^s(\E) = \inf_{\widetilde{E}_i \cap \Omega^c = F_i \cap \Omega^c}  \F^s(\widetilde{\E}).
\]
\end{theorem}

\begin{proof}
The proof follows like that of~\cite{CRS}*{Theorem 3.2} using compactness and  Proposition~\ref{prop:lsc}. 
\end{proof}

\begin{theorem}[Class of local minimizers is compact]
Suppose that $\E^k$ are minimizers for $\F^s$ in $B_1$ and
$\E^k \to \E$ in $L^1_{\text{loc}}(\R^n)$ as $k \to \infty$. 
Then $\E$ is a minimizer for $\F^s$ in $B_1$ and 
\[
\lim_{k \to \infty} \F^s(\E^k,B_1) = \F^s(\E,B_1). 
\]
\end{theorem}

\begin{proof}
The proof follows along the same lines as the proof of~\cite{CRS}*{Theorem 3.3}. 
\end{proof}

We note the following characterization. 

\begin{lemma}\label{lem:localmin}
Let~$\E$ be admissible and $\Omega'\Subset \Omega$. 
If  $\mathbf{F}$ is admissible such that $F_i \cap (\Omega')^c = E_i \cap (\Omega')^c$, $-1 \leq i \leq 1$, for $\Omega' \Subset \Omega$, then
\begin{equation}\label{eq:equaldifference}
\F^s(\E,\Omega)-\F^s(\mathbf{F},\Omega)=\F^{s}(\E, \Omega')  -\F^{s}(\mathbf{F}, \Omega').
\end{equation}
Consequently, $\E$ is a local minimizer for $\F^s$ if and only if, for all $\Omega'\Subset \Omega$,
it holds that
\[
\F^s(\E, \Omega') \leq \F^s(\mathbf{F}, \Omega') \quad \hbox{whenever}~F_i \cap (\Omega')^c = E_i \cap (\Omega')^c~\hbox{for}~-1 \leq i \leq 1. 
\]
\end{lemma}

\begin{proof}
Using Lemma~\ref{lem:omegaomega}, we write
\begin{align*}
\F^s(\E,\Omega)-\F^{s}(\E, \Omega')
	&= \sum_{-1\leq i \leq 1} \alpha_i [\per_{\Omega}^s(E_i)-\per_{\Omega'}^s(E_i)]\\
	&=  \sum_{-1 \leq i \leq 1} \alpha_i  \bigg[L(E_i \cap (\Omega\setminus \Omega'), E_i^c \cap (\Omega')^c)+ L(E_i \cap \Omega^c, E_i^c \cap(\Omega\setminus \Omega')) \bigg].
\end{align*}
If $\mathbf{F}$ is such that $F_i \cap (\Omega')^c = E_i \cap (\Omega')^c$ for $-1 \leq i \leq 1$, then 
\begin{align*}
\F^s(\mathbf{F},\Omega)-\F^{s}(\mathbf{F}, \Omega')
	&=  \sum_{-1 \leq i \leq 1} \alpha_i  \bigg[L(F_i \cap (\Omega\setminus \Omega'), F_i^c \cap (\Omega')^c)+ L(F_i \cap \Omega^c, F_i^c \cap(\Omega\setminus \Omega')) \bigg]\\
	&=  \sum_{-1 \leq i \leq 1} \alpha_i  \bigg[L(E_i \cap (\Omega\setminus \Omega'), E_i^c \cap (\Omega')^c)+ L(E_i \cap \Omega^c, E_i^c \cap(\Omega\setminus \Omega')) \bigg].
\end{align*}
Therefore, 
\[
\F^s(\E,\Omega)-\F^{s}(\E, \Omega') = \F^s(\mathbf{F},\Omega)-\F^{s}(\mathbf{F}, \Omega')
\]
or, equivalently,~\eqref{eq:equaldifference} holds and the lemma is proved. 
\end{proof}

\section{Reducing the nonlocal energy}\label{sec:reduce}

In this section, we address the proof of Theorem~\ref{thm:reduceintro}. We show that, when $\alpha_0<0$,  one can cover an interface  between $E_1$ and $E_{-1}$ with an $\eps>0$ strip of phase~$0$ to reduce the nonlocal energy. Differently from the classical case,
this is not obvious in our setting, due to the presence of remote interactions.

Let us first establish the difference in energy when we enlarge the set $E_0$ in~$\Omega$. 

\begin{lemma}\label{lem:differenceinenergy}
Assume that $\E$ is admissible and that there is some measurable set $A \subset (E_1 \cup E_{-1}) \cap \Omega$.
Set
\[
\widetilde{E}_{-1} = E_{-1} \cap A^c, \quad \widetilde{E}_1 = E_1 \cap A^c, \quad \widetilde{E}_0 = E_0 \cup A.
\]
Then
\begin{align*}
\F^s(\E) -\F^s(\widetilde{\E})
&= \sigma_{-1,0}\left[ L(E_{-1} \cap A, E_{-1}^c) - L(E_{-1} \cap A, E_{-1} \cap A^c) \right]\\
	&\quad + \sigma_{0,1}\left[ L(E_{1} \cap A, E_{1}^c) - L(E_{1} \cap A, E_{1} \cap A^c)\right]\\
	&\quad - 2\alpha_0 \left[ 
	L(E_{-1} \cap A, E_{1} \cap A)
	+L( E_{-1} \cap A, E_1 \cap A^c)
	+L( E_1 \cap A, E_{-1} \cap A^c)
	\right].
\end{align*}
\end{lemma}

\begin{proof}
For ease, we set
\begin{align*}
D
	:=  \F^s(\E) -\F^s(\widetilde{\E})	
	&= \alpha_{-1} \left[ \per_\Omega^s(E_{-1}) -\per_\Omega^s(E_{-1} \cap A^c) \right]\\
	&\quad +\alpha_1 \left[ \per_\Omega^s(E_{1}) -\per_\Omega^s(E_{1} \cap A^c) \right]\\
	&\quad + \alpha_0 \left[  \per_\Omega^s(E_{0})-\per_\Omega^s(E_{0} \cup A)\right].
\end{align*}
By Lemma~\ref{lem:E-difference}, one has
\begin{align*}
D
	&= \alpha_{-1} \left[ L(E_{-1} \cap A, E_{-1}^c) - L(E_{-1} \cap A, E_{-1} \cap A^c) \right]\\
	&\quad +  \alpha_{1} \left[ L(E_{1} \cap A, E_{1}^c) - L(E_{1} \cap A, E_{1} \cap A^c)\right]\\
	&\quad - \alpha_0 \left[  L( A, E_0^c \cap A^c)-L(E_0,  A)\right]\\
	&= (\alpha_{-1} + \alpha_0)\left[ L(E_{-1} \cap A, E_{-1}^c) - L(E_{-1} \cap A, E_{-1} \cap A^c) \right]\\
	&\quad +  (\alpha_{1} +\alpha_0)\left[ L(E_{1} \cap A, E_{1}^c) - L(E_{1} \cap A, E_{1} \cap A^c)\right]\\
	&\quad - \alpha_0 \left[  L( A, E_0^c \cap A^c)
	- L(E_0,  A)\right.\\
	&\qquad\qquad \left.+
	 L(E_{-1} \cap A, E_{-1}^c) - L(E_{-1} \cap A, E_{-1} \cap A^c) 
\right.\\
	&\qquad\qquad \left.	+L(E_{1} \cap A, E_{1}^c) - L(E_{1} \cap A, E_{1} \cap A^c)
	\right].
\end{align*}
To simplify the last two lines, we write
\begin{align*}
 L( A, E_0^c \cap A^c)
	&=  L( A, E_1 \cap A^c) +  L( A, E_{-1} \cap A^c)\\
	&=  L( E_1 \cap A, E_1 \cap A^c) + L( E_{-1} \cap A, E_1 \cap A^c) \\
	&\quad +  L( E_1 \cap A, E_{-1} \cap A^c)+  L( E_{-1} \cap A, E_{-1} \cap A^c)
\end{align*}
and
\begin{align*}
L&(E_{-1} \cap A, E_{-1}^c)+L(E_{1} \cap A, E_{1}^c)\\
&= L(E_{-1} \cap A, E_0)+L(E_{1} \cap A, E_0)
 + L(E_{-1} \cap A, E_{1})+L(E_{1} \cap A, E_{-1})\\
&= L(A, E_0)
 + L(E_{-1} \cap A, E_{1} \cap A)+ L(E_{-1} \cap A, E_{1} \cap A^c)\\
&\quad\qquad+L(E_{1} \cap A, E_{-1} \cap A) + L(E_{1} \cap A, E_{-1}\cap A^c).
\end{align*}
Therefore, we find that
\begin{align*}
D
	&= (\alpha_{-1} + \alpha_0)\left[ L(E_{-1} \cap A, E_{-1}^c) - L(E_{-1} \cap A, E_{-1} \cap A^c) \right]\\
	&\quad +  (\alpha_{1} +\alpha_0)\left[ L(E_{1} \cap A, E_{1}^c) - L(E_{1} \cap A, E_{1} \cap A^c)\right]\\
	&\quad - 2\alpha_0 \left[ 
	L(E_{-1} \cap A, E_{1} \cap A)
	+L( E_{-1} \cap A, E_1 \cap A^c)
	+L( E_1 \cap A, E_{-1} \cap A^c)
	\right],
\end{align*}
and the conclusion follows by recalling ~\eqref{eq:alphasums}.
\end{proof}

It will also be convenient to use the following consequence of Lemma~\ref{lem:differenceinenergy}. 

\begin{corollary}\label{cor:one-sided-strip}
Assume that $\E$ is admissible and that there is some measurable set $A \subset E_1 \cap \Omega$.
Set
\[
\widetilde{E}_{-1} = E_{-1}, \quad \widetilde{E}_1 = E_1 \cap A^c, \quad \widetilde{E}_0 = E_0 \cup A.
\]
Then
\[
\F^s(\E) -\F^s(\widetilde{\E})
=- 2\alpha_0 L(A, E_{-1})
		+ \sigma_{0,1} \left[ L(A, E_1^c) - L(A, E_1 \cap A^c)\right].
\]
\end{corollary}

\begin{proof}
The expression follows from Lemma~\ref{lem:differenceinenergy} since $E_{-1} \cap A = \emptyset$ and $E_1 \cap A= A$. 
\end{proof} 

We now give a more precise statement and proof of Theorem~\ref{thm:reduceintro}. Denote points $x\in \R^n$, with~$n \geq 2$, by $x = (x',x_n) \in \R^{n-1} \times \R$.
Let~$B'_r(x')\subset \R^{n-1}$ denote the ball of radius $r>0$ centered at $x'$ and write $B'_r = B'_r(0)$. 
We provide full details of the case $n=1$ in Appendix~\ref{appendix:1D}.  

\begin{theorem}\label{thm:reduce}
Assume that $s \in (0,1)$ and $\alpha_0<0$. 
For $R,r>0$, assume that $B_r' \times (-R,R) \subset \Omega$. 

Suppose that $\psi: B_r' \to \R$ is a Lipschitz  function such that $\psi(0) = 0$, $\nabla \psi(0) = 0$, and
there is a constant $c_0 \in (0,1)$ such that
\begin{equation}\label{eq:psi}
\|\psi\|_{C^{0,1}(B_r')} \leq \frac{c_0R}{4r}.
\end{equation}
Consider an admissible~$\E$ with
\begin{align*}
\left\{(x',x_n) \in B_r' \times \R~\hbox{s.t.}~\psi(x')< x_n < \psi(x')+\frac{R}{2}\right\} &\subset E_1,\\ 
\left\{(x',x_n) \in B_r' \times \R~\hbox{s.t.}~\psi(x')-\frac{R}{2} < x_n< \psi(x')\right\} &\subset E_{-1}.
\end{align*}

Then, there exists $C_0 = C_0 (n,s,r,R,c_0)>0$  and $\eps_0 = \eps_0(n,s,r,R,c_0,\sigma_{i,j}) \in (0,1)$ such that, for all $\eps \in (0, \eps_0)$, there is an admissible  $\E^\eps$ satisfying 
\[
\F^s(\E) - \F^s(\E^\eps) \geq C_0 |\alpha_0| \eps^{1-s}.
\]
\end{theorem}

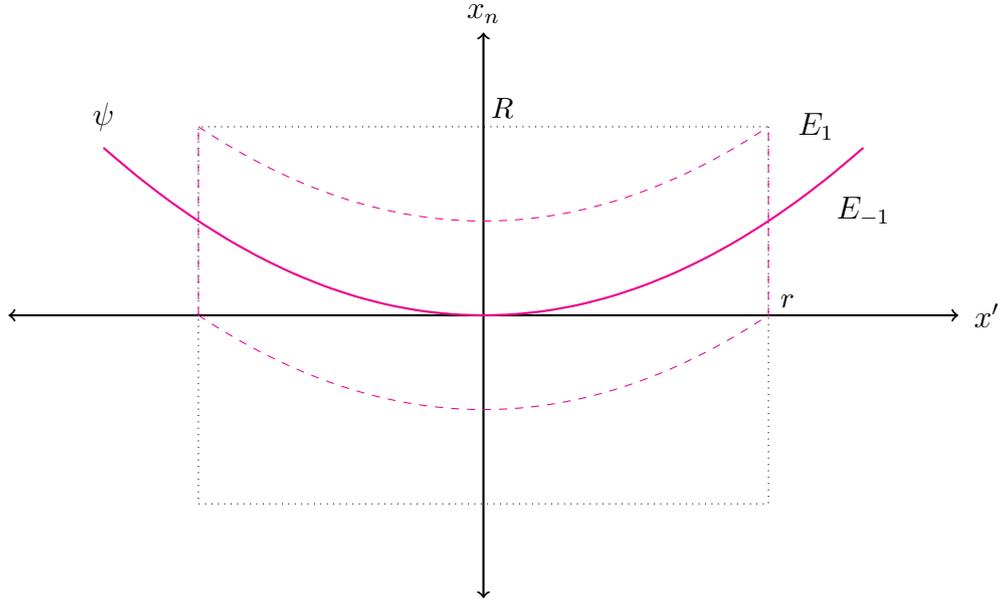
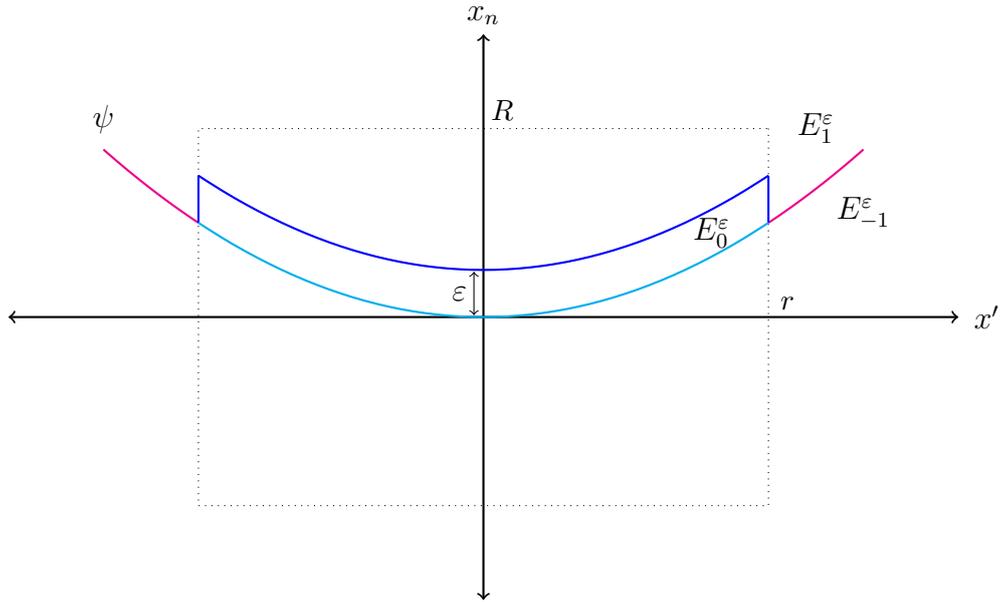
\begin{figure}[htb]
\begin{center}
\subfloat[An admissible $\E$ in $B_r' \times (-R,R)$]{
\begin{tikzpicture}[scale=2.5]
\draw[thick,<->] (-2.5,0)--(2.5,0);
\draw[thick,<->] (0,-1.5)--(0,1.5);
\node[] at (2.65,0) {$x'$};
\node[] at (0,1.6) {$x_n$};
\draw[dotted] (1.5,1)--(1.5,-1)--(-1.5,-1)--(-1.5,1)--(1.5,1);
\node[] at (1.6,.075) {\small$r$};
\node[] at (.1,1.1) {\small$R$};
\draw[magenta,thick] plot[smooth,domain=-2:2] (\x, {2*\x*\x/9} );
\draw[magenta,dashed] plot[smooth,domain=-1.5:1.5] (\x, {2*\x*\x/9+.5} );
\draw[magenta,dashed] plot[smooth,domain=-1.5:1.5] (\x, {2*\x*\x/9-.5} );
\draw[magenta,dashed] (-1.5,0)--(-1.5,1);
\draw[magenta,dashed] (1.5,0)--(1.5,1);
\node[] at (1.75,1) {$E_1$};
\node[] at (2,.55) {$E_{-1}$};
\node[] at (-2, 1.05) {$\psi$};
\end{tikzpicture}
}
\hspace{1.5cm}
\subfloat[The corresponding $\E^\varepsilon$ in $B_r' \times (-R,R)$]{
\begin{tikzpicture}[scale=2.5]
\draw[thick,<->] (-2.5,0)--(2.5,0);
\draw[thick,<->] (0,-1.5)--(0,1.5);
\node[] at (2.65,0) {$x'$};
\node[] at (0,1.6) {$x_n$};
\draw[dotted] (1.5,1)--(1.5,-1)--(-1.5,-1)--(-1.5,1)--(1.5,1);
\node[] at (1.6,.075) {\small$r$};
\node[] at (.1,1.1) {\small$R$};
\draw[magenta,dotted] plot[smooth,domain=-1.5:1.5] (\x, {2*\x*\x/9} );
\draw[magenta,thick] plot[smooth,domain=-2:-1.5] (\x, {2*\x*\x/9} );
\draw[magenta,thick] plot[smooth,domain=1.5:2] (\x, {2*\x*\x/9} );
\draw[blue,thick] plot[smooth,domain=-1.5:1.5] (\x, {2*\x*\x/9+.25} );
\draw[blue,thick] (1.5,.5)--(1.5,.75);
\draw[blue,thick] (-1.5,.5)--(-1.5,.75);
\draw[cyan,thick] plot[smooth,domain=-1.5:1.5] (\x, {2*\x*\x/9} );
\node[] at (1.75,1) {$E_1^\varepsilon$};
\node[] at (2,.55) {$E_{-1}^\varepsilon$};
\node[] at (1.2,.445) { $E_0^\varepsilon$};
\draw[<->] (-.05,.01)--(-.05,.24);
\node[] at (-.125, .125) {$\varepsilon$};
\node[] at (-2, 1.05) {$\psi$};
\end{tikzpicture}
}
\end{center}
\caption{The geometry of Theorem~\ref{thm:reduce}}
\label{fig:strip}
\end{figure}

\begin{proof}
Let~$x' \in B_r'$. By~\eqref{eq:psi}, 
\[
|\psi(x')| 
\leq \frac{c_0 R}{2r}|x'| \leq \frac{c_0 R}{2} \leq \frac{R}{2}. 
\]
For $0 < \eps < \min\{1,R\}/2$, define the sets
\[
A^\eps := \big\{(x',x_n) \in B_r' \times \R~\hbox{s.t.}~\psi(x') < x_n < \psi(x') + \eps\big\} \subset B_r' \times (-R,R) \subset \Omega
\]
and admissible $\E^\eps$ by 
\[
E_1^\eps = E_1 \cap (A^\eps)^c, \quad E_0 = E_0 \cup A^\eps, \quad E_{-1}^\eps = E_{-1}. 
\]
See Figure~\ref{fig:strip}.

By Corollary~\ref{cor:one-sided-strip}, it holds that
\begin{equation}\label{eq:Rndiffnew}
\F^s(\E) - \F^s(\E^\eps)
	= 2|\alpha_0|L(A^\eps, E_{-1})
		- \sigma_{0,1} \left[L(A^\eps, E_1 \cap (A^\eps)^c)- L(A^\eps, E_{1}^c)\right]. 
\end{equation}
In the following, we let $C>0$ denote an arbitrary constant, depending on $n$, $s$, $r$ and $R$, that perhaps changes from line to line. 

\medskip
\underline{\textbf{Step 1}}. 
We begin by showing that there is some $C_0>0$, independent of $\eps$, such that
\begin{equation}\label{eq:below}
L(A^\eps, E_{-1}) \geq C_0 \eps^{1-s}. 
\end{equation}
Since 
\[
\big\{(x',x_n) \in B_r' \times \R~\hbox{s.t.}~\psi(x')-\eps< x_n < \psi(x') \big\}\subset E_{-1},
\]
we find that
\begin{align*}
L(A^\eps, E_{-1})
	&\geq \int_{B_{(1-\eps)r}'} dx' \int_{\psi(x')}^{\psi(x')+\eps} dx_n \int_{B_r'} dy' \int_{\psi(y')-\eps}^{\psi(y')}dy_n \frac{1}{(|x'-y'|^2 + |x_n-y_n|^2)^{\frac{n+s}{2}}}.
\end{align*}
Making the change of variables $x_n \mapsto x_n - \psi'(x')$ and $y_n \mapsto y_n - \psi'(y')$, we find
\begin{align*}
L(A^\eps, E_{-1})
	&\geq \int_{B_{(1-\eps)r}'} dx' \int_{0}^{\eps} dx_n \int_{B_r'} dy' \int_{-\eps}^{0}dy_n \frac{1}{(|x'-y'|^2 + |(x_n-y_n) -(\psi(x') - \psi(y'))|^2)^{\frac{n+s}{2}}}.
\end{align*}
Since $\psi$ is Lipschitz,
\begin{align*}
|x'-y'|^2 + |(x_n-y_n) -(\psi(x') - \psi(y'))|^2
	&\leq  |x'-y'|^2 +2(|x_n-y_n|^2 + |\psi(x') - \psi(y')|^2) \\
	&\leq C (|x'-y'|^2 + |x_n-y_n|^2),
\end{align*}
and we have
\[
L(A^\eps, E_{-1})
	\geq C \int_{B_{(1-\eps)r}'} dx' \int_{0}^{\eps} dx_n \int_{B_r'} dy' \int_{-\eps}^{0}dy_n \frac{1}{(|x'-y'|^2 + |x_n-y_n|^2)^{\frac{n+s}{2}}}.
\]

For $(x',x_n) \in B_{(1-\eps)r}'\times (0,\eps)$, consider 
\[
J(x',x_n) := \int_{-\eps}^0  \int_{B_r'}  \frac{dy' dy_n}{(|x'-y'|^2 + |x_n-y_n|^2)^{\frac{n+s}{2}}}.
\]
With the change of variables $z := (x'-y')/|x_n-y_n|$, we find
\[
J(x',x_n)
	= \int_{-\eps}^0  \frac{1}{|x_n-y_n|^{1+s}}  \left[\int_{B_{\frac{r}{|x_n-y_n|}}'(\frac{x'}{|x_n-y_n|})}  \frac{dz}{(1 + |z|^2)^{\frac{n+s}{2}}} \right] dy_n
\]
Since $\eps <\frac12$, it holds that
\[
B_{\frac{r}{2}}' \subset B_{\frac{r}{|x_n-y_n|}}'\bigg(\frac{x'}{|x_n-y_n|}\bigg),
\]
and we have
\begin{align*}&
J(x',x_n)\geq  \int_{-\eps}^0  \frac{1}{|x_n-y_n|^{1+s}}  \left[\int_{B_{\frac{r}{2}}'}  \frac{dz}{(1 + |z|^2)^{\frac{n+s}{2}}} \right] dy_n\\
&\qquad= C \int_{-\eps}^0 \frac{1}{(x_n-y_n)^{1+s}} \, dy_n 
= C [x_n^{-s}-(x_n+\eps)^{-s}].
\end{align*}
Therefore, 
\begin{equation}\label{eq:lower}
\begin{aligned}
&L(A^\eps, E_{-1})
\geq C \int_0^\eps \int_{B_{(1-\eps)r}'} J(x',x_n) dx' dx_n\\
&\qquad\geq C|B_{\frac{r}{2}}'| \int_0^\eps [x_n^{-s}-(x_n+\eps)^{-s}]  dx_n = C \eps^{1-s},
\end{aligned}
\end{equation}
and~\eqref{eq:below} holds. 

\medskip
\underline{\textbf{Step 2}}. 
We next show that there is some $C_1>0$, independent of $\eps$, such that
\begin{equation}\label{eq:above}
L(A^\eps, E_1 \cap (A^\eps)^c)- L(A^\eps, E_{1}^c) \leq C_1 \eps. 
\end{equation}
First, we use the fact that
\[
E_1 \cap (A^\eps)^c \subset [(B_r')^c \times \R] \cup
\left\{(x',x_n) \in B_r' \times \R~\hbox{s.t.}~x_n < \psi(x')-\frac{R}{2}~\hbox{or}~x_n> \psi(x')+\eps \right\}
\]
to estimate
\begin{align*}
L&(A^\eps, E_1 \cap (A^\eps)^c)\\
	&\leq \int_{B_r'} dx' \int_{\psi(x')}^{\psi(x')+\eps} dx_n
	\int_{(B_r')^c} dy' \int_{-\infty}^{\infty} dy_n \frac{1}{(|x'-y'|^2 +|x_n-y_n|^2)^{\frac{n+s}{2}}}\\
	&\quad + \int_{B_r'} dx' \int_{\psi(x')}^{\psi(x')+\eps} dx_n
	\int_{B_r'} dy' \int_{\{\psi(y')-\frac{R}{2} \leq y_n \leq \psi(y')+\eps\}^c} dy_n \frac{1}{(|x'-y'|^2 +|x_n-y_n|^2)^{\frac{n+s}{2}}}\\
	&=: I_1 + I_2. 
\end{align*}

Furthermore, since
\[
E_1^c = E_0 \cup E_{-1} \supset \left\{(x',x_n)\in B_r' \times \R~\hbox{s.t.}~\psi(x')-\frac{R}{2}< x_n < \psi(x') \right\},
\]
we find
\begin{align*}
L(A^\eps, E_{1}^c)
	&\geq \int_{B_r'} dx' \int_{\psi(x')}^{\psi(x')+\eps} dx_n
	\int_{B_r'} dy' \int_{\psi(y')-\frac{R}{2}}^{\psi(y')} dy_n \frac{1}{(|x'-y'|^2 +|x_n-y_n|^2)^{\frac{n+s}{2}}}\\
	&=: I_3. 
\end{align*}
Combining, we have
\begin{equation}\label{eq:I1I2I3}
L(A^\eps, E_1 \cap (A^\eps)^c)- L(A^\eps, E_{1}^c)
\geq I_1 + I_2 - I_3. 
\end{equation}
We will estimate $I_1$ and $I_2-I_3$ separately from below. 

First, in $I_2$, we make the change of variables $x_n \mapsto x_n - \psi(x')$ and $y_n \mapsto y_n - \psi(y')$ to obtain
\begin{align*}
I_2
&= \int_{B_r'} dx' \int_{0}^{\eps} dx_n \int_{B_r'} dy' \int_{\{-\frac{R}{2} \leq x_n \leq \eps\}^c} dy_n \frac{1}{(|x'-y'|^2 +|(x_n-y_n) - (\psi(x') - \psi(y'))|^2)^{\frac{n+s}{2}}}.
\end{align*}
Similarly in $I_3$, we make the change of variables $x_n \mapsto -x_n+ \psi(x')+\eps$ and $y_n \mapsto -y_n+\psi(y')+\eps$ to obtain
\[
I_3= \int_{B_r'} dx' \int_{0}^{\eps} dx_n
	\int_{B_r'} dy' \int_\eps^{\eps+\frac{R}{2}} dy_n \frac{1}{(|x'-y'|^2 +|(x_n-y_n) - (\psi(x') - \psi(y'))|^2)^{\frac{n+s}{2}}}.
\]
Together, we find
\begin{align*}
&I_2 - I_3\\
	&\quad= 	\int_{B_r'} dx' \int_{0}^{\eps} dx_n \int_{B_r'} dy' \int_{\{-\frac{R}{2} \leq x_n \leq \eps+\frac{R}{2}\}^c} dy_n \frac{1}{(|x'-y'|^2 +|(x_n-y_n) - (\psi(x') - \psi(y'))|^2)^{\frac{n+s}{2}}}\\
	&\quad\leq \int_{B_r'} dx' \int_{0}^{\eps} dx_n \int_{B_r'} dy' \int_{\{|x_n-y_n|>\frac{R}{2}\}} dy_n \frac{1}{|(x_n-y_n) - (\psi(x') - \psi(y'))|^{n+s}}.
\end{align*}

For $x',y' \in B_r'$ and $|x_n-y_n|>\frac{R}{2}$, we use~\eqref{eq:psi} to find
\[
|(x_n-y_n) - (\psi(x') - \psi(y'))|
	\geq |x_n-y_n| - \|\psi\|_{C^{0,1}(B_r')} 2r
	\geq (1-c_0) |x_n-y_n|,
\]
and thus
\begin{equation}\label{eq:I2I3new}
\begin{aligned}
&I_2 - I_3
\leq  C\int_{B_r'} dx' \int_{0}^{\eps} dx_n \int_{B_r'} dy' \int_{\{|x_n-y_n|>\frac{R}{2}\}} dy_n \frac{1}{|x_n-y_n|^{n+s}}\\
&\qquad= 2|B_r'|^2 \int_0^\eps dx_n \int_R^\infty \frac{dt}{t^{n+s}} = C\eps. 
\end{aligned}
\end{equation}

Now, regarding $I_1$, for $x' \in B_r'$ and $y' \in (B_r')^c$, consider
\[
J_1(x',y') := \int_{\psi(x')}^{\psi(x')+\eps}\int_{-\infty}^{\infty} \frac{dy_n dx_n}{(|x'-y'|^2 +|x_n-y_n|^2)^{\frac{n+s}{2}}}.
\]  
Making the change of variables $\xi = x_n/|x'-y'|$ and $\rho = y_n/|x'-y'|$, we have
\begin{align*}
J_1(x',y') 
&= \frac{1}{|x'-y'|^{n+s-2}} \int_{\frac{\psi(x')}{|x'-y'|}}^{\frac{\psi(x')+\eps}{|x'-y'|}}  \int_{-\infty}^\infty \frac{d\rho d\xi}{(1+|\xi-\rho|^2)^{\frac{n+s}{2}}}\\
&\leq  \frac{1}{|x'-y'|^{n+s-2}} \int_{\frac{\psi(x')}{|x'-y'|}}^{\frac{\psi(x')+\eps}{|x'-y'|}} \left[ \int_{\{|\xi-\rho|<1\}} d\rho + \int_{\{|\xi-\rho|>1\}} \frac{d\rho}{|\xi-\rho|^{n+s}}\right] d\xi\\
&= \frac{C\eps}{|x'-y'|^{n+s-1}}.
\end{align*}
Therefore, we have
\begin{equation}\label{eq:I1new}
I_1
	= \int_{B_r'} \int_{(B_r')^c} J_1(x',y')  dy'dx'
	\leq C\eps \int_{B_r'} \int_{(B_r')^c} \frac{dy'dx'}{|x'-y'|^{(n-1)+s}}
	= C \per_{\R^{n-1}}(B_r') \eps = C\eps. 
\end{equation}

By~\eqref{eq:I1I2I3},~\eqref{eq:I2I3new} and~\eqref{eq:I1new}, we have~\eqref{eq:above}.

\medskip
\underline{\textbf{Conclusion}}. We use estimates~\eqref{eq:below} and~\eqref{eq:above} in~\eqref{eq:Rndiffnew} to conclude that
\[
\F^s(\E) - \F^s(\E^\eps) \geq 2|\alpha_0|C_0\eps^{1-s}
		- \sigma_{0,1} C_1 \eps \geq |\alpha_0|C_0 \eps^{1-s}
\]
as long as 
\begin{equation}\label{epsilon0defi}
0 < \eps < \min\left\{\frac12,\frac{R}{2},\left(\frac{|\alpha_0|C_0}{\sigma_{0,1} C_1}\right)^{\frac{1}{s}}\right\} =:\eps_0.\qedhere
\end{equation}
\end{proof}

\begin{remark}
Theorem~\ref{thm:reduce} does not hold in general when $\alpha_0 = 0$. Indeed, we see from formula~\eqref{epsilon0defi} in the proof that $\eps_0 \searrow 0$ as $\alpha_0 \to 0$.
\end{remark}

\section{Uniform estimates for nonlocal area functionals}\label{sec:uniformestimates} 

In this section, we provide uniform estimates for $\per_\Omega^s(A,B)$ and $\per_\Omega^s(E)$ for sufficiently regular sets when $s$ is close to $1$ and consequently obtain Theorem~\ref{thm:uniformintro}. We assume here
that~$n\geq 2$ and provide full details for the case~$n=1$ in Appendix~\ref{appendix:1D}.  

First, we consider the simplest case in which there is no interface between $A$ and $B$. Note that, in contrast to the classical setting, $\per_\Omega^s(A,B) >0$. Nevertheless, as $s\nearrow1^-$, the quantity $(1-s)\per_\Omega^s(A,B)$ converges to $0$, which trivially is the classical area of the interface between~$A$ and~$B$. 

\begin{lemma}\label{lem:separated}
If $A,B \subset \R^n$ are disjoint measurable subsets of $\R^n$ such that, for some $\eps>0$, 
\[
\dist(A\cap \Omega, B) \geq \eps \quad \hbox{and} \quad \dist(A, B \cap \Omega)\geq \eps,
\]
then there is a constant $C$, depending only on $n$ and $\Omega$, such that
\[
\per_\Omega^s(A,B) \leq  \frac{C}{s\eps^s}. 
\]
Consequently,
\[
\lim_{s \nearrow 1} (1-s)\per_\Omega^s(A,B) = 0. 
\]
\end{lemma}

\begin{proof}
Let~$\omega_n$ denote the measure of the unit ball in $\R^n$. 
We simply estimate
\begin{align*}
L^s(A \cap \Omega, B)
	&\leq \int_{A \cap \Omega} \int_{\{|x-y|\geq \eps\}} \frac{dx \, dy}{|x-y|^{n+s}}
	\leq  \frac{n \omega_{n} |\Omega|}{s \eps^{s}},
\end{align*}
and similarly for $L^s(A , B \cap \Omega)$. 
\end{proof}

Next, we will consider the case when $A$ and $B$ share a sufficiently regular interface. 

Let us begin by establishing a refinement of~\cite{CVcalcvar}*{Theorem 1} which says that, for Lipschitz sets, the $s$-perimeter can be made arbitrarily close to the classical perimeter
as~$s\nearrow1$.

Recall the notation~\eqref{eq:delta-nbhd}. 
For $s \in (0,1]$, define
\[
\nu(n,s) := 2 \left(1 - \frac{1}{2^s}\right) (n-1) \omega_{n-1}\int_0^\infty \frac{t^{n-2}}{(1+t^2)^{(n+s)/2}} dt.
\]
It was observed in~\cite{ADPM}*{Equation 24} that $\nu(n,1) = \omega_{n-1}$. 

\begin{theorem}\label{thm:CV}
Assume that $\Omega \subset \R^n$, with~$n \geq 2$, is a bounded domain. Suppose that there exist~$\mu>0$ and~$M \geq 1$ such that
\begin{equation}\label{tab}
\begin{array}{l}
\hbox{for any}~x,y \in \Omega~\hbox{with}~|x-y| \leq \mu,\\
\hbox{there exists a curve of length less than}~M \mu\\
\hbox{which joins}~x~\hbox{and}~y~\hbox{and lies in}~\Omega. 
\end{array}
\end{equation}
Let~$E \subset \R^n$ be such that $\partial E \cap (\Omega \cup \Omega_\delta)$ is Lipschitz for some $0 < \delta<1$.  
 
Then, for any $\eta>0$, there exists $K_{\eta,\delta}>0$ (possibly depending on the size of $\Omega$, $\mu$, $M$,  and the Lipschitz-norm of $\partial E \cap \Omega$) such that
\[
\bigg| \frac{s(1-s)}{\nu(n,s)}\per_\Omega^s(E)-\per_\Omega(E)\bigg| \leq 2 \per_{\Omega_\delta}(E)+ \eta + (1-s)K_{\eta,\delta}. 
\]
\end{theorem}

\begin{proof}
Fix $\eta>0$. 
By~\cite[Lemma 11]{CVcalcvar} (the proof of which holds for Lipschitz regularity in place of $C^{1,\alpha}$), there is a $K_1(\eta)$ such that
\begin{equation}\label{eq:lem11}
\bigg| \frac{s(1-s)}{\nu(n,s)}L^s(E \cap \Omega, E^c \cap \Omega)-\per_\Omega(E)\bigg| \leq \frac{\eta}{3} + (1-s)K_1(\eta). 
\end{equation}
It remains to check the cross terms. 
For this, we follow the proof of~\cite{CVcalcvar}*{Lemma 12}. 

Let~$R >1$ be such that $\Omega \Subset B_R \subset \R^n$.
We estimate
\begin{align*}
&\int_{(E \cap \Omega) \setminus \Omega_\delta} \int_{E^c \cap \Omega^c} \frac{dxdy}{|x-y|^{n+s}} 
\leq \int_{(E \cap \Omega) \setminus \Omega_\delta} \int_{E^c \cap  (B_{2R} \setminus \Omega)} \frac{dxdy}{|x-y|^{n+s}} 
		+ \int_{B_R} \int_{(B_{2R})^c} \frac{dxdy}{|x-y|^{n+s}} \\
&\qquad\leq \delta^{-n-1} |\Omega||B_{2R}|
		+ \int_{B_R} \int_{(B_{2R})^c} \frac{dxdy}{(|y|/2)^{n+s}}\leq C_1(\delta),
\end{align*}
for some constant $C_1(\delta)>0$, and similarly
\begin{align*}
&\int_{E \cap \Omega} \int_{(E^c \cap \Omega^c) \setminus \Omega_\delta} \frac{dxdy}{|x-y|^{n+s}} 
\leq \int_{E \cap \Omega} \int_{(E^c \cap (B_{2R} \setminus \Omega)) \setminus \Omega_\delta} \frac{dxdy}{|x-y|^{n+s}}
		+\int_{B_R} \int_{(B_{2R})^c} \frac{dxdy}{|x-y|^{n+s}} \\
&\qquad\leq \delta^{-n-1} |\Omega| | B_{2R}| +\int_{B_R} \int_{(B_{2R})^c} \frac{dxdy}{(|y|/2)^{n+s}} \leq C_1(\delta).
\end{align*}

On the other hand, by~\cite{CVcalcvar}*{Lemma 11}, there is a $K_2(\eta,\delta)>0$ such that 
\[
\left|\frac{s(1-s)}{\nu(n,s)} \int_{E \cap \Omega_\delta} \int_{E^c \cap \Omega_\delta}\frac{dxdy}{|x-y|^{n+s}}  - \per_{\Omega_\delta}(E)\right| \leq \frac{\eta}{3} + (1-s) K_2(\eta,\delta).
\]
Together, we find
\begin{equation}\label{eq:cross1}
\begin{aligned}
 \frac{s(1-s)}{\nu(n,s)} L^s(E \cap \Omega, E^c \cap \Omega^c)
		&\leq \frac{s(1-s)}{\nu(n,s)} \int_{E \cap \Omega_\delta} \int_{E^c \cap \Omega_\delta}\frac{dxdy}{|x-y|^{n+s}} +  \frac{s(1-s)}{\nu(n,s)}  2C_1(\delta)\\
		&\leq  \per_{\Omega_\delta}(E) + \frac{\eta}{3} + (1-s) K_2(\eta,\delta)+ \frac{s(1-s)}{\nu(n,s)}  2C_1(\delta)\\
		&\leq  \per_{\Omega_\delta}(E) + \frac{\eta}{3} + (1-s) K_3(\eta,\delta)
\end{aligned}
\end{equation}
for some $K_3(\eta,\delta)>0$. 
Replacing $E$ with $E^c$,  we similarly find
\begin{equation}\label{eq:cross2}
 \frac{s(1-s)}{\nu(n,s)} L^s(E^c \cap \Omega, E \cap \Omega^c)
		\leq  \per_{\Omega_\delta}(E) + \frac{\eta}{3} + (1-s) K_3(\eta,\delta).
\end{equation}
Therefore, by~\eqref{eq:lem11},~\eqref{eq:cross1}, and~\eqref{eq:cross2}, we have
\begin{align*}
\bigg| \frac{s(1-s)}{\nu(n,s)}\per_\Omega^s(E)-\per_\Omega(E)\bigg| \leq 
2 \per_{\Omega_\delta}(E) + \eta + (1-s)(K_1(\eta) + 2K_3(\eta,\delta)),
\end{align*}
and the result holds. 
\end{proof}

We obtain a similar result for $\per_{\Omega}^s(A,B)$ when both interfaces $\partial A \cap \partial B$ and $\partial (A \cup B)$ are sufficiently regular. 

\begin{theorem}\label{thm:CVAB}
Assume that $\Omega \subset \R^n$, with~$n \geq 2$, is a bounded domain with Lipschitz boundary. 
Let~$A,B \subset \R^n$ be disjoint subsets of $\R^n$ such that the following conditions hold:
\begin{enumerate}
\item there exists $\mu>0$ and $M \geq 1$ such that
\begin{center}
\begin{tabular}{l}
for any $x,y \in (A \cup B) \cap \Omega$ with $|x-y| \leq \mu$,\\
there exists a curve of length less than $M \mu$\\
which joins $x$ and $y$ and lies in $ (A \cup B) \cap \Omega$. 
\end{tabular}
\end{center}
\item $\partial A \cap \partial B \cap (\Omega \cup \Omega_\delta)$ is Lipschitz for some $0 < \delta<1$.
\end{enumerate} 
Then, for any $\eta>0$, there exists $K_{\eta,\delta}>0$ (possibly depending on the size of $\Omega$, $\mu$, $M$,  and the $C^{0,1}$-norm of $\partial A \cap \partial B \cap \Omega$) such that
\[
\bigg| \frac{s(1-s)}{\nu(n,s)}\per_\Omega^s(A,B)-\per_\Omega(A,B)\bigg| \leq 2 \per_{\Omega_\delta}(A,B)+ \eta + (1-s)K_{\eta,\delta}.
\]
\end{theorem}

\begin{proof}
Let~$\Omega' := (A \cup B) \cap \Omega$. By~\cite[Lemma 11]{CVcalcvar} (with $\Omega'$ in place of $\Omega$ and Lipschitz regularity
in place of $C^{1,\alpha}$), there exists a $K_1(\eta)>0$ such that
\[
\bigg| \frac{s(1-s)}{\nu(n,s)}L^s(A \cap \Omega', A^c \cap \Omega')-\per_{\Omega'}(A)\bigg| \leq \frac{\eta}{3} + (1-s)K_1(\eta),
\]
which can be written equivalently as
\begin{equation}\label{eq:local-int}
\bigg| \frac{s(1-s)}{\nu(n,s)}L^s(A \cap \Omega, B \cap \Omega)-\per_{\Omega}(A,B)\bigg| \leq \frac{\eta}{3} + (1-s)K_1(\eta). 
\end{equation}
It remains to check the cross terms. 

The proof follows along the same lines as that of Theorem~\ref{thm:CV}. Indeed, first, we can find~$C_1(\delta)>0$ such that
\begin{align*}
\int_{(A \cap \Omega) \setminus \Omega_\delta} \int_{B \cap \Omega^c} \frac{dxdy}{|x-y|^{n+s}} 
		\leq C_1(\delta) \quad
\hbox{and} \quad
\int_{A \cap \Omega} \int_{(B \cap \Omega^c) \setminus \Omega_\delta} \frac{dxdy}{|x-y|^{n+s}} 
	 \leq C_1(\delta).
\end{align*} 
Then, as in~\eqref{eq:local-int} with $\Omega_\delta$ in place of $\Omega$, there is a $K_2(\eta,\delta)>0$ such that 
\[
\left|\frac{s(1-s)}{\nu(n,s)} \int_{A \cap \Omega_\delta} \int_{B \cap \Omega_\delta}\frac{dxdy}{|x-y|^{n+s}}  - \per_{\Omega_\delta}(A,B)\right| \leq \frac{\eta}{3} + (1-s) K_2(\eta,\delta).
\]
Together, we get
\begin{equation}\label{eq:cross1AB}
\begin{aligned}
 \frac{s(1-s)}{\nu(n,s)} L^s(A \cap \Omega, B \cap \Omega^c)
		&\leq \frac{s(1-s)}{\nu(n,s)} \int_{A \cap \Omega_\delta} \int_{B \cap \Omega_\delta}\frac{dxdy}{|x-y|^{n+s}} +  \frac{s(1-s)}{\nu(n,s)}  2C_1(\delta)\\
		&\leq  \per_{\Omega_\delta}(A,B) + \frac{\eta}{3} + (1-s) K_3(\eta,\delta)
\end{aligned}
\end{equation}
for some $K_3(\eta,\delta)>0$. 

Interchanging the roles of $A$ and $B$,  we similarly find
\begin{equation}\label{eq:cross2AB}
 \frac{s(1-s)}{\nu(n,s)} L^s(A \cap \Omega^c, B \cap \Omega)
		\leq   \per_{\Omega_\delta}(A,B) + \frac{\eta}{3} + (1-s) K_3(\eta,\delta)
\end{equation}
Therefore, by~\eqref{eq:local-int},~\eqref{eq:cross1AB}, and~\eqref{eq:cross2AB}, we have
\begin{align*}
\bigg| \frac{s(1-s)}{\nu(n,s)}\per_\Omega^s(E)-\per_\Omega(E)\bigg| \leq 
2 \per_{\Omega_\delta}(A,B) + \eta + (1-s)(K_1(\eta) + 3K_3(\eta,\delta)),
\end{align*}
and the result holds. 
\end{proof}

Since $\sigma_{i,j}>0$, we obtain the corresponding result for $\F^s(\E)$ which then implies Theorem~\ref{thm:uniformintro}. 

\begin{corollary}\label{cor:uniform} 
Assume that $\Omega \subset \R^n$, with~$n \geq 2$, is a bounded domain with Lipschitz boundary. Let~$\E$ be admissible such that $\partial E_i \cap (\Omega \cup \Omega_\delta)$ is Lipschitz for some $\delta>0$ and each $-1 \leq i \leq 1$. 

Then, for any $\eta>0$, there exists $K_{\eta,\delta}>0$ (possibly depending on the size of $\Omega$, the Lipschitz norms of the interfaces, and $\sigma_{i,j}$) such that
\[
\bigg| \frac{s(1-s)}{\nu(n,s)}\F^s(\E,\Omega)-\F^1(\E,\Omega)\bigg| \leq 2\F^1(\E, \Omega_\delta) 
+ \eta + (1-s)K_{\eta,\delta}.
\]
\end{corollary}

\begin{remark}
In Theorem~\ref{thm:CVAB}, Corollary~\ref{cor:uniform}, and Theorem~\ref{thm:uniformintro}, the Lipschitz assumption on $\Omega$ can be replaced by~\eqref{tab}. 
\end{remark}

\section{$\Gamma$-convergence}\label{sec:gamma}

This section contains the proofs of Theorems~\ref{thm:s-gamma} and~\ref{thm:localmin}. 

It will be helpful to precisely state the main result from~\cite{ADPM} for reference. 

\begin{theorem}[Theorem 2 in~\cite{ADPM}]\label{thm:ADPM}
The functional $(1-s) \per_\Omega^s$ $\Gamma$-converges to $\omega_{n-1}\per_\Omega$. In particular, if $s_k \nearrow 1$ as $k \to \infty$, then the following holds.  
\begin{enumerate}
\item If $\chi_{E^k} \to \chi_E$ in $L^1_{\text{loc}}(\R^n)$, then 
\[
\liminf_{k \to \infty} (1-s_k)L^{s_k}(E^k \cap \Omega,(E^k)^c \cap \Omega)  \geq \omega_{n-1} \per_\Omega(E).
\]
\item Given a measurable set $E$, there exists $E^k$ such that $\chi_{E^k} \to \chi_E$ in $L^1_{\text{loc}}(\R^n)$ and
\[
\limsup_{k \to \infty} (1-s_k) \per_\Omega^{s_k}(E^k) \leq \omega_{n-1} \per_\Omega(E).
\]
\end{enumerate}
\end{theorem}

Assume that $\alpha_0\leq 0$. 
One can check using Remark~\ref{rem:JAB1} and~\eqref{eq:alpha} that  
\begin{equation}\label{eq:F*2}
\F^*(\E) 
	= \sum_{-1 \leq i \leq 1} \alpha_i^* \per_\Omega (E_i)
	=\sum_{i \in \{-1,1\}} \alpha_i^* \per_\Omega (E_i)
\end{equation}
where 
\[
\alpha_0^* = 0 \quad \hbox{and}\quad \alpha_{i}^* = \alpha_{i} + \alpha_0 = \sigma_{i,0} >0\quad \hbox{for}~i \in  \{-1,1\}. 
\]
Similarly, by Lemma~\ref{lem:JAB}, for $s \in (0,1)$ we have that 
\begin{equation}\label{eq:Fs*}
\begin{aligned}
\F^{s}(\E)
	&= \sum_{i \in \{-1,1\}} \alpha_i \per_\Omega^{s}(E_i)
		+\alpha_0 \per_\Omega^{s}(E_0)\\
	&=  \sum_{i \in \{-1,1\}} (\alpha_i+\alpha_0) \per_\Omega^s(E_i)
		+\alpha_0 [\per_\Omega^{s}(E_0)-\per_\Omega^s(E_1)-\per_\Omega^{s}(E_{-1})]\\
	&= \sum_{i \in \{-1,1\}} \alpha_i^* \per_\Omega^{s}(E_i)
		-2\alpha_0 \per_\Omega^{s}(E_{-1}, E_1).
\end{aligned}
\end{equation}

In light of Theorem~\ref{thm:ADPM}
and~\eqref{eq:F*2}, one expects that the sum over $i \in \{-1,1\}$ in the last line of~\eqref{eq:Fs*} to $\Gamma$-converge to $\F^*$, up to renormalizing. 
We show that the remaining term $-2\alpha_0 (1-s)\per_\Omega^{s}(E_{-1}, E_1)$ is negligible. 
Indeed, since $\alpha_0\leq0$, this term can be dropped when bounding from below. 
When bounding from above, we will build a recovery sequence $\E^k$ for which $E_{-1}^k$ and $E_1^k$ are separated in such a way that, by Lemma~\ref{lem:separated}, this term vanishes in the limit.
However, in our setting, one of the difficulties is ensuring that the recovery sequence is admissible. 

\subsection{Bounding the energy from below}

We establish two $\Gamma$-$\liminf$ inequalities. First, we prove the interior result in Theorem~\ref{thm:s-gamma}. Then, we will prove a similar result up to the boundary and with given exterior data for proving Theorem~\ref{thm:localmin}. 

\begin{proposition}\label{prop:s-liminf}
Let~$\Omega \subset \R^n$ be a bounded domain with Lipschitz boundary and assume that~$\alpha_0\leq0$. 
Let~$s_k \nearrow 1$ as $k \to \infty$. 

If the admissible  $\E$ and $\E^k$ are such that $\E^k \to \E$  in $L^1_{\text{loc}}(\R^n)$ as $k \to \infty$, then 
\begin{equation}\label{eq:s-liminf}
\liminf_{k \to \infty} (1-s_k) \F^{s_k}(\E^k) \geq \omega_{n-1}\F^*(\E). 
\end{equation}
\end{proposition}

We will need the following coercivity result. Recall the notation in~\eqref{eq:Fomega}. 

\begin{lemma}\label{lem:coercivity}
Assume that $\alpha_0\leq0$. 
Let~$s_k \nearrow 1$ as $k \to \infty$. If admissible $\E^k$ are such that
\[
\sup_{k \in \N} (1-s_k)\F^s(\E^k,\Omega') <\infty \quad \hbox{for all}~\Omega' \Subset \Omega,
\]
then $\{E_i^k\}_{k \in \N}$, for~$i \in \{-1,1\}$, is relatively compact in $L^1_{\text{loc}}(\Omega)$ and any limit point $\E = \{E_{-1}, E_0, E_1\}$ is admissible and  such that $E_i$, for~$i \in \{-1,1\}$, have locally finite perimeter in~$\Omega$. 
\end{lemma}

\begin{proof}
It is enough to prove that $\{E_{-1}^k\}_{k \in \N}$  and $\{E_1^k\}_{k \in \N}$ are relatively compact in $L^1_{\text{loc}}(\Omega)$ and any limit points $E_{-1}$ and $E_1$, respectively, have locally finite perimeter in~$\Omega$. 

Fix $i \in \{-1,1\}$ and let $\Omega' \Subset \Omega$. 
By~\eqref{eq:Fs*} (with $\Omega'$ in place of $\Omega$) and since $\alpha_0\leq0$, we have that
\[
 \sup_{k \in \N} (1-s_k) \alpha_i^* \per_{\Omega'}^{s_k}(E_i^k) 
\leq \sup_{k \in \N} (1-s_k) \F^s(\E^k,\Omega') <\infty.
\]
By~\cite{ADPM}*{Theorem 1}, it holds that $\{E_i^k\}_{k \in \N}$ is relatively compact in $L^1_{\text{loc}}(\Omega)$ and any limit point~$E_i$ has locally finite perimeter in~$\Omega$. 
\end{proof}

We now proceed with the proof of the $\liminf$-inequality. 

\begin{proof}[Proof of Proposition~\ref{prop:s-liminf}]
We may assume that 
\[
\liminf_{k \to \infty}(1-s_k) \F^{s_k}(\E^k) < \infty
\]
since otherwise the inequality is trivial. Consequently, by Lemma~\ref{lem:coercivity}, the sets $E_i$ have finite perimeter in~$\Omega$. 

Recalling~\eqref{eq:Fs*} and that $\alpha_0\leq0$, we find that
\[
\F^{s_k}(\E^k)\geq \sum_{i \in \{-1,1\}} \alpha_i^* \per_\Omega^{s_k}(E_i^k). 
\]
Since $\alpha_i^*>0$ for $i \in \{-1,1\}$, we use Theorem~\ref{thm:ADPM} to find
\begin{align*}
&\liminf_{k \to \infty} (1-s_k)\F^{s_k}(\E^k) 
\geq \sum_{i \in \{-1,1\}} \alpha_i^* \liminf_{k \to \infty} (1-s_k) \per_\Omega^{s_k}(E_i^k)\\
&\qquad \geq \sum_{i \in \{-1,1\}} \alpha_i^* \omega_{n-1} \per_\Omega(E_i)
	= \omega_{n-1}\F^*(\E),
\end{align*}
which proves~\eqref{eq:s-liminf}. 
\end{proof}

\subsubsection{Given exterior data}

We now prove a $\liminf$-inequality up to the boundary with prescribed exterior boundary conditions. 
First, we consider the fractional perimeter functional. 

Recalling~\eqref{eq:delta-nbhd} , for $\delta>0$, set
\[
(\Omega_\delta^+)^o = \{x \in \Omega^c~\hbox{s.t.}~0 < \dist(x, \Omega) <\delta\}. 
\]

\begin{lemma}\label{lem:bdry-liminf-oneset}
Assume that $\Omega\subset\R^n$ is a bounded domain with Lipschitz boundary. 
Suppose that~$E \subset \R^n$ is such that 
$\per_{(\Omega_\delta^+)^o}(E)<\infty$  and $\partial E \cap (\Omega_\delta^+)^o$ is Lipschitz for some $\delta \in (0,1)$. 

Let~$s_k \nearrow 1$ as $k \to \infty$. 
If $\E^k$ are such that $E^k \cap \Omega^c = E \cap \Omega^c$ and $\chi_{E^k} \to \chi_{E}$ in $L_{\text{loc}}^1(\R^n)$, then
\[
\liminf_{k \to \infty} (1-s_k) \per_{\Omega}^{s_k}(E^k) \geq \omega_{n-1} \per_{\overline{\Omega}}(E).
\]
\end{lemma}

\begin{proof}
By Theorem~\ref{thm:ADPM}, it holds that
\[
\liminf_{k \to \infty} (1-s_k)L^{s_k}(E^k \cap (\Omega \cup \Omega_\delta),(E^k)^c \cap (\Omega \cup \Omega_\delta)) \geq \omega_{n-1} \per_{\Omega \cup \Omega_\delta} (E). 
\]
We observe that
\begin{align*}
&L^{s_k}(E^k \cap (\Omega \cup \Omega_\delta),(E^k)^c \cap (\Omega \cup \Omega_\delta))\\
	&\quad= L^{s_k}(E^k \cap \Omega, (E^k)^c \cap \Omega)
	 +L^{s_k}(E^k \cap \Omega_\delta^+, (E^k)^c \cap \Omega)
	+ L^{s_k}(E^k \cap \Omega, (E^k)^c \cap \Omega_\delta^+)\\
	&\qquad +L^{s_k}(E \cap \Omega_\delta^+,  E^c \cap \Omega_\delta^+) \\
	&\quad\leq \per_{\Omega}^{s_k}(E^k) + L^{s_k}(E \cap \Omega_\delta^+,  E^c \cap \Omega_\delta^+). 
\end{align*}
By~\cite{CVcalcvar}*{Lemma 11} (for which one may assume Lipschitz regularity in place of $C^{1,\alpha}$), we have
\[
\lim_{k \to \infty} (1-s_k)L^{s_k}(E \cap \Omega_\delta^+,  E^c \cap \Omega_\delta^+) = \omega_{n-1} \per_{(\Omega_\delta^+)^o}(E) <\infty, 
\]
hence
\begin{align*}
\omega_{n-1} \per_{\Omega \cup \Omega_\delta}(E)
 	\leq \liminf_{k \to \infty} (1-s_k) \per_{\Omega}^{s_k}(E^k) + \omega_{n-1}\per_{(\Omega_\delta^+)^o}(E). 
\end{align*}
Therefore, we conclude that
\[
\liminf_{k \to \infty} (1-s_k) \per_{\Omega}^{s_k}(E^k)  \geq \omega_{n-1} [\per_{\Omega \cup \Omega_\delta}(E) - \per_{(\Omega_\delta^+)^o}(E)] = \omega_{n-1} \per_{\overline{\Omega}}(E).\qedhere
\]
\end{proof}

We then obtain the analogous result for $\F^s$. 

\begin{proposition}\label{prop:bdry-liminf}
Let~$\Omega \subset \R^n$ be a bounded domain with Lipschitz boundary and assume that~$\alpha_0\leq0$. 
Suppose that $\E$ is admissible and such that, for $i \in \{-1,1\}$, 
$\per_{(\Omega_\delta^+)^o}(E_i)<\infty$  and $\partial E_i \cap (\Omega_\delta^+)^o$ is Lipschitz for some $\delta\in(0,1)$. 

Let~$s_k \nearrow 1$ as $k \to \infty$. 
If admissible $\E^k$ are such that $E_i^k \cap \Omega^c = E_i \cap \Omega^c$ for each $-1 \leq i \leq 1$ and $\E^k \to \E$ in $L^1_{\text{loc}}(\R^n)$ as $k \to \infty$, then 
\begin{equation*}
\liminf_{k \to \infty} (1-s_k) \F^{s_k}(\E^k,\Omega) \geq \omega_{n-1}\F^*(\E, \overline{\Omega}). 
\end{equation*}
\end{proposition}

\begin{proof}
The proof follows along the same lines as that of  Proposition~\ref{prop:s-liminf} (using Lemma~\ref{lem:bdry-liminf-oneset} instead of Theorem~\ref{thm:ADPM}). 
\end{proof}

\subsection{Bounding the energy from above}

We prove two $\Gamma$-$\limsup$ inequalities. First, we prove the interior result in Theorem~\ref{thm:s-gamma}. Then, we will prove a similar result with given exterior data needed for Theorem~\ref{thm:localmin}. 

\begin{proposition}\label{prop:s-limsup}
Assume that $\alpha_0\leq0$, and let $s_k \nearrow 1$ as $k \to \infty$. 
Let~$\E$ be admissible. 

Then there exists a sequence of admissible $\E^k$ such that $\E^k \to \E$ in $L^1_{\text{loc}}(\R^n)$ as $k \to \infty$  and
\begin{equation}\label{eq:s-limsup}
\limsup_{k \to \infty} (1-s_k) \F^{s_k}(\E^k) \leq  \omega_{n-1}\F^*(\E). 
\end{equation}
\end{proposition}

\begin{proof}
We may assume that $\F^*(\E)<\infty$; otherwise there is nothing to show. 

Fix $i \in \{-1,1\}$. For $\eps>0$, define
\[
\widetilde{E}_i^\eps = \{ x \in E_i~\hbox{s.t.}~\dist(x, \partial E_i ) \geq 2\eps\},
\]
and notice that $\chi_{\widetilde{E}_i^\eps} \to \chi_{E_i}$ as $\eps \searrow 0$ in $L^1_{\text{loc}}(\R^n)$ and 
\[
\per_\Omega(E_i) \geq \lim_{\eps \searrow 0} \per_\Omega(\widetilde{E}_i^\eps). 
\]
By~\cite{Modica}*{Lemma 1}, 
there exist open sets $E_i^\eps$, with smooth boundaries in $\R^n$ and finite (classical) perimeter in~$\Omega$, such that $\chi_{E_i^\eps} \to \chi_{E_i}$ as $\eps \searrow 0$ in $L^1_{\text{loc}}(\R^n)$, 
\begin{equation}\label{eq:eps-properties}
\mathcal{H}^{n-1}(\partial E_i^\eps \cap \partial \Omega) = 0, \quad 
\per_\Omega(E_i) = \lim_{\eps \searrow 0} \per_\Omega(E_i^\eps), \quad \hbox{and} 
\quad \dist(E_i^\eps, \partial E_i) \geq \eps. 
\end{equation}
For $i=0$, we define 
\[
E_0^\eps = \R^n \setminus (E_1^\eps \cup E_{-1}^\eps), 
\]
so the admissible $\E^\eps$ are such that $\E^\eps \to \E$ in   
$L^1_{\text{loc}}(\R^n)$ as $\eps \searrow 0$. 

Looking at the energy for $\E$, we find
\begin{align*}
& \omega_{n-1}\F^*(\E)
= \sum_{i \in \{-1,1\}} \alpha_i^* \omega_{n-1} \per_\Omega(E_i)\\
&\qquad \geq \lim_{\eps \searrow 0} \sum_{i \in \{-1,1\}} \alpha_i^* \omega_{n-1} \per_\Omega(E_i^\eps)
	=   \lim_{\eps \searrow 0} \omega_{n-1}\F^*(\E^\eps).
\end{align*}
By Remark~\ref{rem:uniform} and~\eqref{eq:eps-properties}, we have that
\[
\lim_{k \to \infty} \big|(1-s_k) \F^{s_k}(\E^\eps) -\omega_{n-1}\F^*(\E^\eps)\big| = 0,
\]
which implies 
\begin{equation}\label{eq:finallimit}
 \omega_{n-1}\F^*(\E)
 	\geq \lim_{\eps \searrow 0} \omega_{n-1}\F^*(\E^\eps)
	=  \lim_{\eps \searrow 0}\, \limsup_{k \to \infty}(1-s_k) \F^{s_k}(\E^\eps).
\end{equation}
The conclusion of the proposition now follows from standard arguments; we will write the details for completeness. 

Let~$\ell \in \N$. By~\eqref{eq:finallimit}, there is $\eps_\ell>0$ such that 
\[
 \omega_{n-1}\F^*(\E) + \frac{1}{\ell} \geq \limsup_{k \to \infty} (1-s_k) \F^{s_k}(\E^{\eps_\ell}).
\]
Next, there is $K_\ell \in \N$ such that for all $k \geq K_\ell$,
\[
 \omega_{n-1}\F^*(\E) + \frac{2}{\ell} \geq  (1-s_k) \F^{s_k}(\E^{\eps_\ell}).
\]
We may assume that $K_\ell < K_{\ell+1}$. Then, there is an invertible map~$\psi:\N \to \N$ such that~$K_\ell = \psi(\ell)$. 
Consider the map $\varphi: \N \to \N$ such that $\varphi(k) = K_\ell$ when $k \in [K_\ell, K_{\ell+1})$. Then, upon relabeling,
\[
E^k_i := E^{\eps_\ell}_i
	\quad \hbox{where}~\ell = \psi^{-1}(K_\ell) = \psi^{-1}(\varphi(k))~\hbox{is determined by}~k,
\]
the admissible $\E^k = \{E_{-1}^k, E_0^k, E_{-1}^k\}$ satisfies
\[
\F^*(\E) + \frac{2}{\ell} \geq  (1-s_k) \F^{s_k}(\E^{k}) \quad \hbox{for all}~k \geq K_\ell. 
\]
Since $\chi_{E_i^k} \to \chi_{E_i}$ as $k \to \infty$ in $L^1_{\text{loc}}(\R^n)$ for each $-1 \leq i \leq 1$, this proves the proposition. 
\end{proof}

\subsubsection{Given exterior data}\label{sec:bdry-limsup}

We now prove a $\limsup$-inequality with prescribed exterior boundary conditions. 

\begin{proposition}\label{prop:bdry-limsup}
Let~$\Omega \subset \R^n$ be a bounded domain with $C^1$ boundary and assume  that $\alpha_0\leq0$. 
Let~$s_k \nearrow 1$ as $k \to \infty$. 

If the admissible $\E$ is polyhedral in $\Omega_{\delta_0}^+$ for some $\delta_0>0$ and transversal to $\partial \Omega^+$, then there exists a sequence of admissible $\E^k$ such that 
$E_i^k \cap \Omega^c = E_i \cap \Omega^c$  for each $-1 \leq i \leq 1$, $\E^k \to \E$ in $L^1_{\text{loc}}(\R^n)$ as $k \to \infty$ in $L^1_{\text{loc}}(\R^n)$,  and
\begin{equation}\label{eq:bdry-limsup}
\limsup_{k \to \infty} (1-s_k) \F^{s_k}(\E^k) \leq  \omega_{n-1}\F^*(\E). 
\end{equation}
\end{proposition}

\begin{proof}
We may assume that $\F^*(\E)<\infty$; otherwise there is nothing to show. In particular, $\per_{\Omega}(E_i)<\infty$ for each $-1\leq i \leq 1$. 

Fix $\eta>0$. 
By~\cite{Cesaroni-Novaga}*{Theorem 2.9} (see also~\cite{BCG}), there exists admissible $\E^\eta$ which is polyhedral in $\Omega \cup \Omega_{\delta_0}$, transversal to $\partial \Omega$, such that $E_i^\eta \cap \Omega^c = E_i \cap \Omega^c$ for each $-1 \leq i \leq 1$, $\E^\eta \to \E$ in~$L^1(\R^n)$ as $\eta \searrow 0$, and 
\[
\lim_{\eta \searrow 0} \F^*(\E^\eta) =\F^*(\E). 
\]
(Note that the proof of~\cite{Cesaroni-Novaga}*{Theorem 2.9} only uses that $\E$ is polyhedral in $\Omega_{\delta_0}^+$ for some $\delta_0>0$.)

Now let $\eps>0$ and define
\[
A^\eps := \big\{x \in \Omega~\hbox{s.t.}~\dist(x, \partial E_1^\eta \cap \partial E_{-1}^\eta) < \eps\big\}. 
\]
Define also
\[
E_1^{\eta,\eps} := E_1^\eta \cap (A^\eps)^c, \quad
E_{-1}^{\eta,\eps} := E_{-1}^\eta \cap (A^\eps)^c, \quad
E_{0}^{\eta,\eps} := E_{0}^\eta \cup A^\eps.
\]
For $\eps>0$ small, the admissible $\E^{\eta,\eps}$ is such that $E_i^{\eps, \eta}$ is Lipschitz in $\Omega \cup \Omega_{\delta_0}$ and 
$E_i^{\eps,\eta} \cap \Omega^c = E_i \cap \Omega^c$ for each $-1 \leq i \leq 1$, 
and $\E^{\eta,\eps} \to \E^\eta$ in $L^1_{\text{loc}}(\R^n)$ as $\eps \searrow 0$,  and
\[
\lim_{\eta \searrow 0}\,  \lim_{\eps \searrow 0} \F^*(\E^{\eta,\eps}) =\F^*(\E). 
\]

By Theorem~\ref{thm:CV}, we have, for all $\delta \in (0,\delta_0)$ and $i \in \{-1,1\}$,
\[
\lim_{\eps \searrow 0} \,\lim_{k \to \infty} \bigg|\omega_{n-1} \per_\Omega(E_i^{\eta,\eps})-(1-s_k) \per_{\Omega}^{s_k}(E_i^{\eta,\eps})\bigg| 
\leq \lim_{\eps \searrow 0} 2 \per_{\Omega_\delta}(E_i^{\eps, \eta})
=  2 \per_{\Omega_\delta}(E_i^{\eta}).
\]
Sending $\delta \searrow 0$ and using that $E_i^\eta$ is transversal to $\partial \Omega$, we have
\[
\lim_{\eta \searrow 0}\, \lim_{\eps \searrow 0} \, \lim_{k \to \infty} \bigg|\omega_{n-1} \per_\Omega(E_i^{\eta,\eps})-(1-s_k) \per_{\Omega}^{s_k}(E_i^{\eta,\eps})\bigg|  = 0. 
\]
By construction, we also have that
\[
\dist(E_1^{\eta,\eps} \cap \Omega, E_{-1}^{\eta,\eps})\geq \eps \quad \hbox{and}  \quad
\dist(E_{-1}^{\eta,\eps} \cap \Omega, E_{1}^{\eta,\eps}) \geq \eps,
\]
so, by Lemma~\ref{lem:separated}, we get
\[
\lim_{k \to \infty}(1-s_k)\per_{\Omega}^{s_k}(E_1^{\eta,\eps}, E_{-1}^{\eta,\eps}) = 0. 
\]
Therefore, using~\eqref{eq:Fs*}, we obtain
\begin{align*}
 \omega_{n-1}\F^*(\E)
	&= \lim_{\eta \searrow 0}\, \lim_{\eps \searrow 0} \sum_{i \in \{-1,1\}} \alpha_i^* \omega_{n-1} \per_\Omega(E_i^{\eta,\eps})\\
	&\geq \lim_{\eta \searrow 0}\, \lim_{\eps \searrow 0} \, \limsup_{k \to\infty}(1-s_k)\sum_{i \in \{-1,1\}} \alpha_i^*   \per_\Omega^{s_k}(E_i^{\eta,\eps})\\
	&= \lim_{\eta \searrow 0}\, \lim_{\eps \searrow 0}\,  \limsup_{k \to\infty} (1-s_k) \left[ F^{s_k}(\E^{\eta,\eps}) + 2\alpha_0 \per_\Omega^{s_k}(E_{-1}^{\eta,\eps}, E_{1}^{\eta,\eps})\right]\\
	&= \lim_{\eta \searrow 0}\, \lim_{\eps \searrow 0} \, \limsup_{k \to\infty} (1-s_k) F^{s_k}(\E^{\eta,\eps}),
\end{align*}
and the proposition holds by extracting infinitesimal sequences $\eta = \eta(k)$ and $\eps = \eps(k)$ and defining $\E^k := \E^{\eta(k), \eps(k)}$. 
\end{proof}

\subsection{Proof of Theorem~\ref{thm:s-gamma}}

The result follows directly from Propositions~\ref{prop:s-liminf} and~\ref{prop:s-limsup}. \hfill \qed

\subsection{Proof of Theorem~\ref{thm:localmin}}

First note that Propositions ~\ref{prop:bdry-liminf} and~\ref{prop:bdry-limsup} together prove that $(1-s) \F^s$ $\Gamma$-converges to $\omega_{n-1} \F^*$ where the functionals $\F^s$ and $\F^*$ are defined only on admissible~$\E$ such that $E_i \cap \Omega^c = \widetilde{E}_i \cap \Omega^c$ for each $-1 \leq i \leq 1$. 

\medskip

We need to show that sequences $\F^s(\E^s)$ of minimizers are equibounded.

\begin{lemma}\label{lem:min-sequences}
Let~$\Omega \subset \R^n$ be a bounded domain with Lipschitz boundary and assume that $\alpha_0\leq0$. 
Let~$s_k \nearrow 1$ as $k \to \infty$.
Assume that $\E^k$ are local minimizers of $\F^{s_k}$ in~$\Omega$ and $\E^k \to \E$ in $L^1_{\text{loc}}(\R^n)$ as $k \to \infty$. Then
\begin{equation}\label{eq:finiteenergy}
\limsup_{k \to \infty} (1-s_k) \F^{s_k}(\E^k, \Omega') <\infty \quad \hbox{for all}~ \Omega' \Subset \Omega.
\end{equation}
\end{lemma}

\begin{proof}
Let~$\delta>0$ be small. 
Define the sequence of admissible $\mathbf{F}^k$ by
\[
F_i^k= E_i^k \cap (\Omega^c \cup \Omega_\delta)~\hbox{for}~i \in \{-1,1\} 
\quad \hbox{and}\quad
F_0^k= \R^n \setminus (F_{-1}^k \cup F_1^k).
\]
Using that $\E^k$ is a local minimizer, Lemma~\ref{lem:localmin}, and that $\alpha_0\leq0$, we estimate
\[
\F^{s_k}(\E^k, \Omega \setminus \Omega_\delta)
	\leq \F^{s_k}(\mathbf{F}^k, \Omega \setminus \Omega_\delta)
	\leq \sum_{i \in \{-1,1\}} \alpha_i \per_{\Omega \setminus \Omega_\delta}^{s_k}(F_i^k).
\]
With Lemma~\ref{lem:omegaomega}, we proceed as in the proof of~\cite{ADPM}*{Theorem 3} to find that
\begin{align*}
&\limsup_{k \to \infty} (1-s_k) \F^{s_k}(\E^k, \Omega \setminus \Omega_\delta^-) 
\leq \sum_{i \in \{-1,1\}} \alpha_i \limsup_{k \to \infty} (1-s_k)\per_{\Omega \setminus \Omega_\delta^-}^{s_k}(F_i^k) \\
&\qquad \leq \sum_{i \in \{-1,1\}} \alpha_i \limsup_{k \to \infty} (1-s_k)\per_{\Omega}^{s_k}(F_i^k) 
	<\infty.
\end{align*}
Therefore,~\eqref{eq:finiteenergy} holds for $\Omega' \subset \Omega \setminus \Omega_\delta$ and, consequently, for any $\Omega' \Subset \Omega$. 
\end{proof}

\begin{proof}[Proof of Theorem~\ref{thm:localmin}]
By Lemmata~\ref{lem:min-sequences} and~\ref{lem:coercivity}, up to a subsequence, 
there exists an admissible~$\E$ such that $\F^*(\E)<\infty$, 
$E_i \cap \Omega^c = \widetilde{E}_i \cap \Omega^c$ for each $-1 \leq i \leq 1$, and 
$\E^k \to \E$ in $L^1_{\text{loc}}(\R^n)$ as $k \to \infty$. 

Let~$\mathbf{F}$ be such that $F_i \cap \Omega^c = \widetilde{E}_i \cap \Omega^c$. By Propositions~\ref{prop:bdry-liminf} and~\ref{prop:bdry-limsup}, there are admissible $\mathbf{F}^{k}$ 
such that $F_i \cap \Omega^c = \widetilde{E}_i \cap \Omega^c$ for each $-1 \leq i \leq 1$, $\mathbf{F}^k \to \mathbf{F}$ in $L^1_{\text{loc}}(\R^n)$ as $k \to \infty$, and
\[
\lim_{k \to \infty}(1-s_k) \F^{s_k}(\mathbf{F}^{k}) = \omega_{n-1} \F^*(\mathbf{F}).
\]
By Proposition~\ref{prop:bdry-liminf} and the minimality of $\E^k$, it follows that
\[
\omega_{n-1} \F^*(\E) 
\leq \liminf_{k \to \infty}(1-s_k)\F^{s_k}(\E^k) 
\leq \lim_{k \to \infty}(1-s_k) \F^{s_k}(\mathbf{F}^{k}) = \omega_{n-1} \F^*(\mathbf{F}).
\]
In particular, $\E$ is a minimizer of $\F^*$. Taking $\mathbf{F} = \E$ gives~\eqref{eq:recoverenergy}. 
\end{proof}

\begin{appendix}
\section{Analysis in one-dimension}\label{appendix:1D}

We record and prove the one-dimensional counterparts of Theorems~\ref{thm:reduce}, \ref{thm:CV}, and~\ref{thm:CVAB}. 
Moreover, we prove Theorem~\ref{thm:s-gamma}  up to the boundary of $\Omega \subset \R$. 

\subsection{Reducing the nonlocal energy in 1D}

First, we prove a one-dimensional version of Theorem~\ref{thm:reduce}.

\begin{theorem}\label{thm:reduce1D}
Suppose that $\alpha_0<0$ and that $\Omega \subset \R$ is a bounded domain with $0 \in \Omega$. For~$r>0$, consider an admissible~$\E$ with
\[
0 \in \partial E_1 \cap \partial E_{-1} \cap \Omega, \quad
(0,r)  \subset E_1, \quad (-r,0) \subset E_{-1}.
\]

Then, there are $C_0 = C_0(s) >0$ and $\eps_0 = \eps_0(s,r,\dist(0,\partial\Omega),\sigma_{i,j})>0$
such that, for all~$\eps \in (0,\eps_0)$, there is an admissible~$\E^\eps$ satisfying 
\[
\F^s(\E) - \F^s(\E^\eps) \geq C_0|\alpha_0| \eps^{1-s}. 
\]
\end{theorem}

\begin{proof}
For $0 < \eps < \min\{\dist(0,\partial \Omega), r\}$, define the set $A^\eps = (0,\eps) \subset \Omega$ and the admissible $\E^\eps$ by 
\[
E_1^\eps = E_1 \cap (A^\eps)^c, \quad E_0 = E_0 \cup A^\eps, \quad E_{-1}^\eps = E_{-1}. 
\]
By Corollary~\ref{cor:one-sided-strip}, it holds that
\begin{equation}\label{eq:1Ddiff}
\F^s(\E) - \F^s(\E^\eps)
	= 2|\alpha_0|L(A^\eps, E_{-1})
		- \sigma_{0,1} \left[L(A^\eps, E_1 \cap (A^\eps)^c)- L(A^\eps, E_{1}^c)\right]. 
\end{equation}

Let us first estimate $L(A^\eps, E_{-1})$ from below. Since $(-\eps,0) \subset (-r,0) \subset E_{-1}$, we find that
\begin{equation}\label{eq:1Dbelow}
\begin{aligned}
& L(A^\eps, E_{-1})
\geq \int_0^\eps \int_{-\eps}^0 (x-y)^{-1-s} dydx\\
&\qquad= \frac{1}{s} \int_0^\eps \left[ x^{-s} - (x+\eps)^{-s}\right] dx 
	= \frac{2- 2^{1-s}}{s(1-s)} \eps^{1-s} =: C_0 \eps^{1-s}. 
\end{aligned}
\end{equation}

Next, we estimate $L(A^\eps, E_1 \cap (A^\eps)^c)-L(A^\eps, E_{1}^c)$ from above. Note that 
\[
E_1 \cap (A^\eps)^c \subset [(-\infty,-r) \cup (\eps,\infty)] \quad \hbox{and} \quad
E_{1}^c = E_0 \cup E_{-1} \supset (-r,0). 
\] 
Thus, we have
\begin{align*}
&L(A^\eps, E_1 \cap (A^\eps)^c)-L(A^\eps, E_{1}^c)\\
	&\qquad\leq \int_0^\eps \left[ \int_{-\infty}^{-r} \frac{1}{|x-y|^{1+s}} dy + \int_\eps^\infty \frac{1}{|x-y|^{1+s}} dy \right] dx
		- \int_0^\eps \int_{-r}^0 \frac{1}{|x-y|^{1+s}} dydx. 
\end{align*}
In the last integral, we make the change of variables $x \mapsto -x+\eps$ and $y \mapsto -y+\eps$ to obtain
\begin{equation}\label{eq:1Dabove}
\begin{aligned}
&L(A^\eps, E_1 \cap (A^\eps)^c)-L(A^\eps, E_{1}^c)\\
	&\leq  \int_0^\eps \left[ \int_{-\infty}^{-r} \frac{1}{|x-y|^{1+s}} dy + \int_\eps^\infty \frac{1}{|x-y|^{1+s}} dy \right] dx
		- \int_0^\eps \int_\eps^{r+\eps} \frac{1}{|x-y|^{1+s}} dydx\\
	&=  \int_0^\eps \left[ \int_{-\infty}^{-r} \frac{1}{|x-y|^{1+s}} dy + \int_{r+\eps}^\infty \frac{1}{|x-y|^{1+s}} dy \right] dx\\
	&\leq \int_0^\eps \int_{\{|x-y|>\frac{r}{2}\}} \frac{1}{|x-y|^{1+s}} dydx 
	= \frac{2^{1+s}}{s r^s}\eps =: C_1 \eps. 
\end{aligned}
\end{equation}

Therefore, using the estimates~\eqref{eq:1Dbelow} and~\eqref{eq:1Dabove} in~\eqref{eq:1Ddiff}, the result follows.  
\end{proof}

\begin{remark}
In the setting of Theorem~\ref{thm:reduce1D} with $r = +\infty$, one can show explicitly that the admissible $\E^\eps$ in the proof satisfy
\[
\F^s(\E) - \F^s(\E^\eps) = \frac{2|\alpha_0|}{s(1-s)} \eps^{1-s}.
\]
\end{remark}

\subsection{Uniform estimates in 1D}

Next, we establish Theorems~\ref{thm:CV} and~\ref{thm:CVAB} in dimension one. Let us begin with a uniform estimate for the fractional perimeter functional. 

\begin{theorem}\label{thm:CV1D}
Assume that $\Omega = (a,b) \subset \R$ is a bounded, open interval. 
Let~$E \subset \R$ be such that $\per_{\Omega \cup \Omega_{\delta_0}}(E)<\infty$ for some small $\delta_0 \in (0,|\Omega|/4)$. 
Set
\[
r := \frac12\min\big\{|x-y|~\hbox{s.t.}~x,y \in (\partial E \cap  \Omega) \cup \partial \Omega,~x\not=y\big\}.
\]
 
Then, there exists $K>0$, depending on $r$, $\delta_0$,  $\dist(a, \partial E)$, $\dist(b, \partial E)$ and  the size of $\Omega$, such that 
\[
\left| s(1-s) \per_\Omega^s(E) - \per_{\overline{\Omega}}(E)\right|
\leq |(2-2^{1-s})r^{1-s}-1|\per_\Omega(E) + (1-s)K. 
\]
Consequently, 
\[
\lim_{ s \nearrow 1} s(1-s) \per_\Omega^s(E)= \per_{\overline{\Omega}}(E).
\]
\end{theorem}

\begin{proof}
Consider first the interactions in $\Omega \times \Omega$. 
If $\per_{\Omega}(E) = 0$, then either~$\Omega = E \cap \Omega$ or~$\Omega = E^c \cap \Omega$ and we trivially have that
\[
s(1-s)L(E \cap \Omega, E^c \cap \Omega) =0 =  \per_{\Omega}(E).
\]

Suppose that $\per_{\Omega}(E) =: N \in \N $. Then, there exists $x_j \in \Omega$ for $j=1,\dots, N$ such that
\[
\partial E \cap \Omega = \{x_j\}_{j=1}^N \quad \hbox{and } \; x_j < x_{j+1} ~\hbox{for}~j=1,\dots, N-1. 
\]
We estimate
\begin{align*}
&|s(1-s)L(E \cap \Omega, E^c \cap \Omega)-\per_{\Omega}(E)|\\
	&\leq \left|s(1-s)\sum_{j=1}^N \int_{x_j}^{x_j+r} \int_{x_j-r}^{x_j} \frac{dydx}{|x-y|^{1+s}}-N\right| + s(1-s)\int_{\{(x,y) \in (E \cap \Omega) \times (E^c \cap \Omega) : |x-y|>r\}} \frac{dydx}{|x-y|^{1+s}} \\ 
	&\leq \left| (2-2^{1-s})r^{1-s}-1\right|N + (1-s)\frac{4|\Omega|}{r^s}.
\end{align*}
Therefore, 
\begin{equation}\label{eq:1Domegaomega}
|s(1-s)L(E \cap \Omega, E^c \cap \Omega)-\per_{\Omega}(E)|= |(2-2^{1-s})r^{1-s}-1|\per_\Omega(E) + (1-s)K_1
\end{equation}
with $K_1 >0$ depending on $r$ and the size of $\Omega$. 

Now consider the cross-terms. 
Since $\per_{\Omega \cup \Omega_{\delta_0}}(E)<\infty$, 
we fix~$\delta \in (0,\min\{1,\delta_0\})$ such that 
\[
\dist(a, (\partial E \cup \partial \Omega) \setminus \{a\}) \geq \delta \quad \hbox{and} \quad \dist(b, (\partial E \cup \partial \Omega) \setminus \{b\}) \geq \delta.
\]
Since $\delta < |\Omega|/4$, we have
\begin{equation}\label{eq:1D-cross1}
L( (a, a+\delta), (-\infty, a-\delta)\cup (b,\infty))
	\leq  \int_{\Omega} \int_{\{|x-y|\geq \delta\}} \frac{dydx}{|x-y|^{1+s}}
	= \frac{2|\Omega|}{s \delta^{s}}
\end{equation}
and similarly
\begin{equation}\label{eq:1D-cross2}
L( (b-\delta, b), (-\infty,a)\cup (b+\delta,\infty))
	\leq  \frac{2|\Omega|}{s \delta^{s}}.
\end{equation}
Thus, it is enough to consider interactions in $(a,a+\delta) \times(a-\delta,a) \subset \Omega \times \Omega^c$. Interactions in~$(b-\delta,b) \times (b,b+\delta)\subset \Omega \times \Omega^c$ will follow similarly. 

Set 
\[
\rho_a :=  \dist(a, \partial E).
\]
We will show that
\begin{equation}\label{eq:a-cross}\begin{split}
&\Big|s(1-s) [L(E \cap (a,a+\delta)), E^c \cap (a-\delta,a))+L(E \cap (a-\delta,a)), E^c \cap (a,a+\delta))]\\
&\qquad\qquad-\per_{\{a\}}(E) \Big|\\&\qquad
	\leq (1-s)K_2
\end{split}\end{equation}
for some $K_2>0$ depending on $\delta$ and $\rho_a$. We break the argument into cases. 

\medskip

\emph{Case 1}. If $\rho_a\geq \delta$, then we trivially have
\[
\big|s(1-s) L(E \cap (a,a+\delta)), E^c \cap (a-\delta,a))\big|=0
\]
and
\[
\big|s(1-s) L(E \cap (a-\delta,a)), E^c \cap (a,a+\delta))\big|=0.
\]
Thus,~\eqref{eq:a-cross} holds with $K_2=0$. 

\medskip

\emph{Case 2}. Suppose that $\rho_a \in (0,\delta)$. 
Without loss of generality, suppose that $a \in E$, so that~$(a-\rho_a, a+\rho_a) \subset E$. 
Then, we have 
\begin{align*}
s(1-s)L(E \cap (a,a+\delta)), E^c \cap (a-\delta,a))
	&\leq s(1-s)\int_a^{a+\delta} \int_{a-\delta}^{a-\rho_a} \frac{dydx}{(x-y)^{1+s}} \\
	&= (\delta +\rho_a)^{1-s}-\rho_a^{1-s} + (1-2^{1-s})\delta^{1-s} \\
	&\leq (1-s) \frac{\delta}{\rho_a^{s}}. 
\end{align*}
Similarly, 
\begin{align*}
s(1-s)L(E \cap (a-\delta,a)), E^c \cap (a,a+\delta))
	&\leq s(1-s) \int_{a-\delta}^a \int_{a+\rho_a}^{a+\delta}\frac{dydx}{(y-x)^{1+s}} 
	\leq (1-s) \frac{\delta}{\rho_a^{s}}.
\end{align*}
Therefore,~\eqref{eq:a-cross} holds with $K_2 >0$ depending on $\delta$ and $\rho_a$. 

\medskip

\emph{Case 3}. Suppose that $\rho_a=0$. In this case, $a \in \partial E$. 
Without loss of generality, suppose that~$(a, a+\delta) \subset E$ and $(a-\delta, a) \subset E^c$. 
Then, we have
\begin{align*}
\big|&s(1-s) L(E \cap (a,a+\delta)), E^c \cap (a-\delta,a))-1\big|\\
	&= \big|s(1-s) \int_{a}^{a+\delta} \int_{a-\delta}^a \frac{dydx}{|x-y|^{1+s}} -1\big|
	\leq  \big|(2-2^{1-s})\delta^{1-s} -1\big|
	 \leq 2(1-s)
\end{align*}
and
\[
s(1-s) L(E \cap (a-\delta,a)), E^c \cap (a,a+\delta)) = 0.
\]

\medskip

Combining Cases 1-3, we obtain~\eqref{eq:a-cross}.
By~\eqref{eq:1D-cross1},~\eqref{eq:1D-cross2}, and~\eqref{eq:a-cross}, there is some $K_3>0$ such that
\begin{equation}\label{eq:1Dcrossfinal}
|s(1-s)[L(E \cap \Omega, E^c \cap \Omega^c) + L(E \cap \Omega^c, E^c \cap \Omega) - \per_{\partial \Omega}(E)| \leq (1-s)K_3.
\end{equation}

The estimates~\eqref{eq:1Domegaomega} and~\eqref{eq:1Dcrossfinal} prove the result. 
\end{proof}

We also obtain uniform estimates for $\per_{\Omega}^s(A,B)$ with $s$ close to $1$. Since the proof is similar to that of Theorem~\ref{thm:CV1D} (see also Section~\ref{sec:uniformestimates}), we omit the details. 

\begin{theorem}\label{thm:CVAB1D}
Assume that $\Omega = (a,b) \subset \R$ is a bounded, open interval. 
Let~$A,B \subset \R$ be disjoint such that $\per_{\Omega \cup \Omega_{\delta_0}}(A,B)<\infty$ for some $\delta_0 \in (0,|\Omega/4)$. Set
\[
r := \frac12\min\big\{|x-y|~\hbox{s.t.}~x,y \in (\partial A \cap \partial B \cap  \Omega) \cup \partial \Omega,~x\not=y\big\}.
\]
 
Then, there exists $K>0$, depending on $r$, $\delta_0$,  $\dist(a, \partial A \cap \partial B)$, $\dist(b, \partial A \cap \partial B)$, $\dist(\partial A, \partial B)$ and  the size of $\Omega$, such that 
\[
\left| s(1-s) \per_\Omega^s(A,B) - \per_{\overline{\Omega}}(A,B)\right|
\leq |(2-2^{1-s})r^{1-s}-1|\per_\Omega(A,B) + (1-s)K. 
\]
\end{theorem}

\subsection{$\Gamma$-convergence in 1D up to the boundary}

Theorem~\ref{thm:s-gamma}  holds in dimension one as written. 
Here, we prove a refined result up to the boundary.  

\begin{theorem}\label{thm:1D-gamma}
Let~$\Omega \subset \R$ be an open, bounded interval and assume that $\alpha_0\leq0$. 

Then $(1-s) \F^s(\E,\Omega)$ $\Gamma$-converges to $\F^*(\E,\overline{\Omega})$ within the class of admissible $\E$  such that 
$\per_{(\Omega_{\delta}^+)^o}(E_i)<\infty$ for some $\delta>0$ and for each $-1 \leq i \leq 1$.
\end{theorem}

Since the $\liminf$-inequality in Proposition~\ref{prop:bdry-liminf} holds in dimension one,  it is enough to build a recovery sequence. 

\begin{proposition}\label{prop:1D-limsup}
Let~$\Omega \subset \R$ be an open, bounded interval and assume that $\alpha_0<0$. 
Suppose that $\E$ is admissible and such that 
$\per_{(\Omega_\delta^+)^o}(E_i)<\infty$ for each $-1 \leq i \leq 1$.

Let~$s_k \nearrow 1$ as $k \to \infty$. 
Then, there exists a sequence of admissible $\E^k$ such that 
$E_i^k \cap \Omega^c = E_i \cap \Omega^c$ for each $-1 \leq i \leq 1$, $\E^k \to \E$ in $L^1_{\text{loc}}(\R)$ as $k \to \infty$, 
and
\begin{equation}\label{eq:1D-limsup}
\limsup_{k \to \infty} (1-s_k) \F^{s_k}(\E^k,\Omega) \leq  \F^*(\E,\overline{\Omega}). 
\end{equation}
\end{proposition}

\begin{proof}
Assume that $\F^*(\E,\overline{\Omega})<\infty$; otherwise there is nothing to show. 

Suppose that there is a finite sequence $\{x_j\}_{j=1}^N \subset\overline{\Omega}$ such that
\[
\partial E_1 \cap \partial E_{-1} \cap \overline{\Omega} = \{x_j\}_{j=1}^N,
\]
and  there is $\eps_0 \leq |\Omega|/2$ such that
\[
\dist(x_j, x_\ell)> 2\eps_0 \quad \hbox{for all}~j\not=\ell. 
\]

For $\eps \in (0,\eps_0)$, define $A^\eps \subset \Omega$ by 
\[
A^\eps := \bigcup_{j=1}^N (x_j-\eps, x_j+\eps)  \cap \Omega.
\]
If $\per_{\overline{\Omega}}(E_1, E_{-1}) = 0$, then take $A_\eps = \emptyset$.
 
Also, we define the admissible $\E^\eps$ by
\[
E_1^\eps = E_1 \cap (A^\eps)^c, \quad 
E_{-1}^\eps = E_{-1} \cap (A^\eps)^c, \quad 
E_0^\eps = E_0 \cup A^\eps. 
\]
Notice that $E_i^\eps \cap \Omega^c = E_i \cap \Omega^c$ for each $-1 \leq i \leq 1$,  $\E^\eps \to \E$  in $L^1_{\text{loc}}(\R)$ as $\eps \searrow 0$,
 and 
\[
\sum_{i \in \{-1,1\}} \alpha_i^*\per_{\overline{\Omega}}(E_i) \geq \lim_{\eps \searrow 0}\sum_{i \in \{-1,1\}} \alpha_i^*\per_{\overline{\Omega}}(E_i^\eps). 
\]
Proceeding as in the proof of Proposition~\ref{prop:bdry-limsup}, we use Theorem~\ref{thm:CV1D},~\eqref{eq:Fs*} and Lemma~\ref{lem:separated} to complete the proof. 
\end{proof}

\end{appendix}
\section*{Acknowledgments}

This work has been supported by the Australian Future Fellowship
FT230100333 ``New perspectives on nonlocal equations''
and by the Australian Laureate Fellowship FL190100081 ``Minimal surfaces, free boundaries and partial differential equations.''

\bibliographystyle{imsart-number}
\bibliography{gamma}

\begin{bibdiv}
\begin{biblist}

\bib{MR1612250}{article}{
      author={Alberti, Giovanni},
      author={Bellettini, Giovanni},
       title={A nonlocal anisotropic model for phase transitions. {I}. {T}he
  optimal profile problem},
        date={1998},
        ISSN={0025-5831},
     journal={Math. Ann.},
      volume={310},
      number={3},
       pages={527\ndash 560},
         url={https://doi.org/10.1007/s002080050159},
      review={\MR{1612250}},
}

\bib{Almgren}{article}{
      author={Almgren, Frederick~J., Jr.},
       title={Existence and regularity almost everywhere of solutions to
  elliptic variational problems with constraints},
        date={1976},
     journal={Mem. Amer. Math. Soc.},
      volume={4},
      number={165},
       pages={viii+199},
}

\bib{AmbrosioBraides}{article}{
      author={Ambrosio, Luigi},
      author={Braides, Andrea},
       title={Functionals defined on partitions in sets of finite perimeter.
  {II}.\ {S}emicontinuity, relaxation and homogenization},
        date={1990},
     journal={J. Math. Pures Appl. (9)},
      volume={69},
      number={3},
       pages={307\ndash 333},
}

\bib{ADPM}{article}{
      author={Ambrosio, Luigi},
      author={De~Philippis, Guido},
      author={Martinazzi, Luca},
       title={Gamma-convergence of nonlocal perimeter functionals},
        date={2011},
     journal={Manuscripta Math.},
      volume={134},
      number={3-4},
       pages={377\ndash 403},
}

\bib{MR1051228}{article}{
      author={Baldo, Sisto},
       title={Minimal interface criterion for phase transitions in mixtures of
  {C}ahn-{H}illiard fluids},
        date={1990},
        ISSN={0294-1449},
     journal={Ann. Inst. H. Poincar\'{e} C Anal. Non Lin\'{e}aire},
      volume={7},
      number={2},
       pages={67\ndash 90},
         url={https://doi.org/10.1016/S0294-1449(16)30304-3},
      review={\MR{1051228}},
}

\bib{BCG}{article}{
      author={Braides, Andrea},
      author={Conti, Sergio},
      author={Garroni, Adriana},
       title={Density of polyhedral partitions},
        date={2017},
     journal={Calc. Var. Partial Differential Equations},
      volume={56},
      number={2},
       pages={Paper No. 28, 10},
}

\bib{Brezis}{article}{
      author={Brezis, Kh.},
       title={How to recognize constant functions. {A} connection with
  {S}obolev spaces},
        date={2002},
     journal={Uspekhi Mat. Nauk},
      volume={57},
      number={4(346)},
       pages={59\ndash 74},
}

\bib{CRS}{article}{
      author={Caffarelli, Luis},
      author={Roquejoffre, Jean-Michel},
      author={Savin, Ovidiu},
       title={Nonlocal minimal surfaces},
        date={2010},
     journal={Comm. Pure Appl. Math.},
      volume={63},
      number={9},
       pages={1111\ndash 1144},
}

\bib{CVcalcvar}{article}{
      author={Caffarelli, Luis},
      author={Valdinoci, Enrico},
       title={Uniform estimates and limiting arguments for nonlocal minimal
  surfaces},
        date={2011},
     journal={Calc. Var. Partial Differential Equations},
      volume={41},
      number={1-2},
       pages={203\ndash 240},
}

\bib{Cesaroni-Novaga}{article}{
      author={Cesaroni, Annalisa},
      author={Novaga, Matteo},
       title={Nonlocal minimal clusters in the plane},
        date={2020},
     journal={Nonlinear Anal.},
      volume={199},
       pages={111945, 11},
}

\bib{Colombo-Maggi}{article}{
      author={Colombo, Maria},
      author={Maggi, Francesco},
       title={Existence and almost everywhere regularity of isoperimetric
  clusters for fractional perimeters},
        date={2017},
     journal={Nonlinear Anal.},
      volume={153},
       pages={243\ndash 274},
}

\bib{PAPERONE}{article}{
      author={Dipierro, Serena},
      author={Patrizi, Stefania},
      author={Valdinoci, Enrico},
      author={Vaughan, Mary},
       title={Asymptotic expansion of a nonlocal phase transition energy},
        date={2024},
     journal={arXiv preprint arXiv:2409.06215},
       pages={pp 57},
}

\bib{PAPERTWO}{article}{
      author={Dipierro, Serena},
      author={Valdinoci, Enrico},
      author={Vaughan, Mary},
       title={A threshhold for higher-order asymptotic development of genuinely
  nonlocal phase transition energies},
        date={2025},
     journal={To appear in Proc. Amer. Math. Soc.},
       pages={pp 15},
}

\bib{Giusti}{book}{
      author={Giusti, Enrico},
       title={Minimal surfaces and functions of bounded variation},
   publisher={Department of Pure Mathematics, Australian National University,
  Canberra},
        date={1977},
}

\bib{MR1833000}{article}{
      author={Kinderlehrer, David},
      author={Liu, Chun},
       title={Evolution of grain boundaries},
        date={2001},
        ISSN={0218-2025},
     journal={Math. Models Methods Appl. Sci.},
      volume={11},
      number={4},
       pages={713\ndash 729},
      review={\MR{1833000}},
}

\bib{Leonardi}{article}{
      author={Leonardi, Gian~Paolo},
       title={Infiltrations in immiscible fluids systems},
        date={2001},
     journal={Proc. Roy. Soc. Edinburgh Sect. A},
      volume={131},
      number={2},
       pages={425\ndash 436},
}

\bib{Modica}{article}{
      author={Modica, Luciano},
       title={The gradient theory of phase transitions and the minimal
  interface criterion},
        date={1987},
     journal={Arch. Rational Mech. Anal.},
      volume={98},
      number={2},
       pages={123\ndash 142},
}

\bib{Morgan}{article}{
      author={Morgan, Frank},
       title={Lowersemicontinuity of energy clusters},
        date={1997},
     journal={Proc. Roy. Soc. Edinburgh Sect. A},
      volume={127},
      number={4},
       pages={819\ndash 822},
}

\bib{SV-gamma}{article}{
      author={Savin, Ovidiu},
      author={Valdinoci, Enrico},
       title={{$\Gamma$}-convergence for nonlocal phase transitions},
        date={2012},
     journal={Ann. Inst. H. Poincar\'e Anal. Non Lin\'eaire},
      volume={29},
      number={4},
       pages={479\ndash 500},
}

\bib{SV-dens}{article}{
      author={Savin, Ovidiu},
      author={Valdinoci, Enrico},
       title={Density estimates for a variational model driven by the
  {G}agliardo norm},
        date={2014},
     journal={J. Math. Pures Appl. (9)},
      volume={101},
      number={1},
       pages={1\ndash 26},
}

\bib{2024arXiv241101586S}{article}{
      author={Solci, Margherita},
       title={{Higher-order non-local gradient theory of phase-transitions}},
        date={2024},
     journal={arXiv preprint arXiv:2411.01586},
       pages={pp 33},
}

\bib{MR4743481}{article}{
      author={Solci, Margherita},
       title={Nonlocal-interaction vortices},
        date={2024},
        ISSN={0036-1410},
     journal={SIAM J. Math. Anal.},
      volume={56},
      number={3},
       pages={3430\ndash 3451},
         url={https://doi.org/10.1137/23M1563438},
      review={\MR{4743481}},
}

\bib{White}{article}{
      author={White, Brian},
       title={Existence of least-energy configurations of immiscible fluids},
        date={1996},
     journal={J. Geom. Anal.},
      volume={6},
      number={1},
       pages={151\ndash 161},
}

\end{biblist}
\end{bibdiv}

\end{document}